\documentclass[12pt,a4paper]{scrartcl}

\usepackage[utf8]{inputenc}
\usepackage[british]{babel}
\usepackage{amsmath,amsthm,amssymb,amsfonts}
\usepackage{mathtools}
\mathtoolsset{centercolon}
\usepackage{hyperref}
\usepackage{float,caption}
\usepackage[nameinlink, noabbrev]{cleveref}
\usepackage{geometry,enumitem}
\usepackage{dsfont}
\usepackage{textpos}
\usepackage{thm-restate}
\usepackage{titling}
\usepackage{soul}
\usepackage{xcolor}
\usepackage{macros}
\usepackage{tikz}
\usetikzlibrary{shapes.geometric,hobby,calc}
\usetikzlibrary{shapes,decorations}
\usetikzlibrary{decorations.markings,intersections}
\usepackage{tikzmacros}

\makeatletter
\@addtoreset{claim}{theorem}
\makeatother

\hypersetup{
	colorlinks=true,
	linkcolor=myBlue,
	citecolor=myBlue,
	urlcolor=myLightBlue,
	bookmarksopen=true,
	bookmarksnumbered,
	bookmarksopenlevel=2,
	bookmarksdepth=3
}

\title{Decomposition of (infinite) digraphs along directed 1-separations}
\author{Nathan Bowler
\and
    Florian Gut
\and
    Meike Hatzel
\and 
    Ken-ichi Kawarabayashi
\and 
    Irene Muzi
\and 
    Florian  Reich}
\predate{}
\date{}
\postdate{}

\preauthor{}
\DeclareRobustCommand{\authorthing}{
	\begin{center}
	    Nathan Bowler -- Universität Hamburg, Germany\\
	    \href{mailto:nathan.bowler@uni-hamburg.de}{nathan.bowler@uni-hamburg.de}
	    
	    \medskip
	    Florian Gut\thanks{F.~Gut was supported by the Japan Society for the Promotion of Science (JSPS) under grant no. SP22319.} -- Universität Hamburg, Germany\\
	    \href{mailto:florian.guth@uni-hamburg.de}{florian.gut@uni-hamburg.de}
	    
	    \medskip
	    Meike Hatzel\thanks{M.~Hatzel was supported by the Federal Ministry of Education and
Research (BMBF) and by a fellowship within the IFI programme of the German Academic Exchange Service (DAAD).} -- National Institute of Informatics, Tokyo, Japan\\
	    \href{mailto:research@meikehatzel.com}{research@meikehatzel.com}\\
	    
	    \medskip
	    Ken-ichi Kawarabayashi\thanks{K.~ Kawarabayashi was supported by JSPS Kakenhi JP18H0529, by JSPS Kakenhi JP20A402 and by JSPS Kakenhi 22H05001.} -- National Institute of Informatics, Tokyo, Japan\\
        \href{mailto:k_keniti@nii.ac.jp}{k\_keniti@nii.ac.jp}\\
	    
        \medskip 
        Irene Muzi\thanks{For the purposes of open access, the author has applied a CC BY public copyright licence to any author accepted manuscript version arising from this submission.} -- Birkbeck, University of London, United Kingdom\\
        \href{mailto:irene.muzi@gmail.com}{irene.muzi@gmail.com}
        
	    \medskip
	    Florian Reich -- Universität Hamburg, Germany\\
	    \href{mailto:florian.reich@uni-hamburg.de}{florian.reich@uni-hamburg.de}
\end{center}}
\author{\authorthing}
\postauthor{}    
    
\begin{document}
\maketitle

\begin{textblock}{20}(-1.8, 5.2)
   \includegraphics[width=80px]{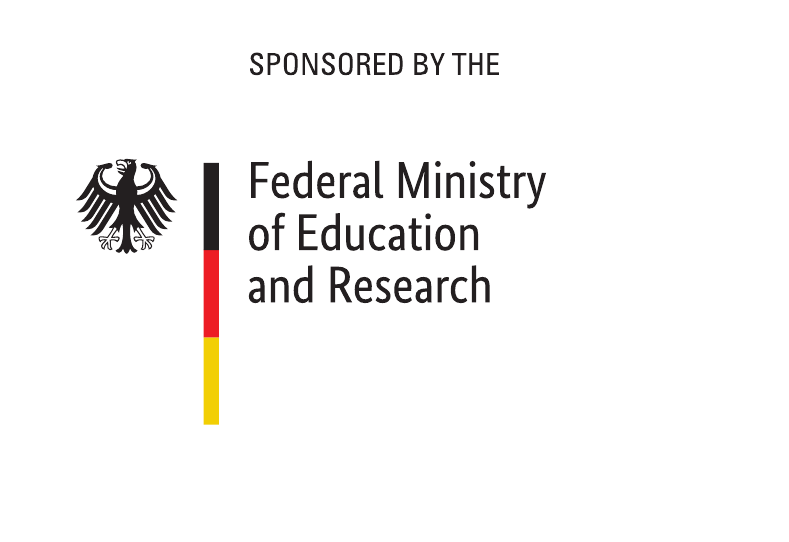}%
\end{textblock}

\begin{abstract}
    We introduce \emph{torsoids}, a canonical structure in matching covered graphs, corresponding to the bricks and braces of the graph.
    This allows a more fine-grained understanding of the structure of finite and infinite directed graphs with respect to their 1-separations.

    \textbf{Keywords:} directed separations, digraphs, infinite digraphs, decomposition, perfect matching, tight cuts, torsoids
\end{abstract}

\newpage

\section{Introduction}
\label{sec:introduction}

Separations in (di-)graphs are a structural property that has long been of interest for mathematicians.
In undirected graphs for small $k$, there are simple and canonical combinatorial structures displaying all $k$-separations of a given $k$-connected undirected graph and the relationships between them.
For $k=0$ this is trivial; it is simply the decomposition of the graph into its connected components.
For $k = 1$ this is the block-cut decomposition \cite[Chapter 3.1]{diestel17} and for $k=2$ it is the Tutte decomposition \cite{Tutte66}.
Recent work of Carmesin and Kurkofka has developed an analogous such structure with $k=3$ \cite{Carmesin23}.

For digraphs the situation is very different, with almost nothing being known.
In the case $k=0$ we again have the decomposition of the digraph into its strongly connected components, together with the partial order on those components given by the reachability relation.
However, already for $k=1$ previous progress on this problem was limited to a partial result by Lov\'asz~\cite{Lovasz1987}\footnote{In fact Lov\'asz, like us, works in the more general framework of matching covered graphs.}.

In order to explain what Lov\'asz' result means for digraphs, we need to consider an operation on strongly connected digraphs.
Let $D$ be a digraph, and let $(A,B)$ be a directed 1-separation of $D$ with separator $v.$
Let $D_A$ be the digraph obtained from $D$ by contracting all of $B$ onto $v$ and let $D_B$ be the digraph obtained by contracting all of $A$ onto $v.$
We say that $D_A$ and $D_B$ are obtained from $D$ by \emph{pulling it apart} along $(A,B).$ 

Now imagine taking a strongly connected digraph, pulling it apart along some directed 1-separation, then continuing to pull those parts apart for as long as possible until you are left with a list of strongly 2-connected digraphs.
In this context, Lov\'asz' result says that it does not matter in what order you carry out this process or which 1-separations you pick, you will always end up with the same list up to list rearrangement and digraph isomorphism.

To understand the limitations of this result, let's compare it with the Tutte decomposition.
An analogous decomposition procedure would be the following.
Given a 2-connected \emph{undirected} graph $G$ and a 2-separation $(A,B)$ of $G$ with separator $\{x,y\}$, we could define $G_A$ and $G_B$ to be the graphs obtained from $G[A]$ and $G[B]$ respectively by adding the edge $xy$ to each of them.
We could once again consider a procedure of repeatedly pulling a graph apart along such separations until we are left with a list of 3-connected graphs, and it follows from the Tutte decomposition that you will always end up with the same list of 3-connected graphs up to list rearrangement and graph isomorphism.

However, this adaption to 3-connected graphs yields a weaker result than the original Tutte decomposition in a number of ways.
First, the elements of the list are only given up to isomorphism, but in the Tutte decomposition they are given as the torsos of a canonical tree-decomposition of $G.$
Second, the adapted decomposition procedure cuts up the 2-connected graph too much.
In the Tutte decomposition, some of the torsos are cycles, and these are sensibly not cut up any further because there is no way to do so canonically.
But the adapted decomposition procedure will happily cut up an $n$-cycle into $n-2$ triangles, and in so doing lose the possibility of finding canonical representatives in the original graph.
Third, the Tutte decomposition provides a global tree structure along which the parts are arranged, and this information is not contained in the list of 3-connected graphs obtained by the adapted decomposition procedure.

Lov\'asz' result has the same 3 limitations.
The elements of the list are only given up to isomorphism, not as canonical structures within the original digraph.
It cuts any directed $n$-cycle up into $n-1$ directed 2-cycles, although there is no hope of finding canonical representatives for these in the original graph.
Finally, it does not provide any global structure along which the parts are arranged.

In this paper, we resolve the first two limitations of Lov\'asz' result, but not the third.
More precisely, we find canonical structures, which we call torsoids, within the digraph corresponding to the parts into which it would be cut by Lov\'asz' procedure.
However, we do not cut up directed cycles any further, since there would be no way to find canonical structures representing the parts.
As an illustration of the canonicity of our results and the deeper structural understanding they provide, we are able to extend Lov\'asz' results to infinite digraphs.

Having discussed the limitations of Lov\'asz' result, it is worth noting one major advantage.
He was able to prove his result in the more general context of {\em matching covered graphs} (that is, connected graphs such that every edge is contained in at least one perfect matching).
There is a well-understood correspondence between strongly connected digraphs and bipartite matching-covered graphs~\cite{mccuaig2000evendicycles,rst1999pfaffianorientations}, under which directed 1-separations correspond to tight cuts.
Lov\'asz' result deals with the process of cutting up (not necessarily bipartite) matching covered graphs along their tight cuts.

Since all our arguments also work just as easily in this more general context, and since there has been a renewed interest in the structure of matching covered graphs in recent years \cite{Norine2006,Norine2007,hatzel2019,Giannopoulou21}, we also phrase all of our results in this paper in these more general terms.

Our paper is structured as follows.
After introducing some preliminaries in \cref{sec:preliminaries} and exploring the basic structure of infinite matching covered graphs and their tight cuts in \cref{sec:even_odd} and \cref{sec:crossings}, we introduce our new canonical torsoids in \cref{sec:torsoids}.
After exploring their relation to tight cuts in \cref{sec:relation_tigh_sets_torsoids}, we generalise Lov\'asz' result to infinite matching covered graphs in \cref{sec:torsos_to_torsoids}.
   
\section{Preliminaries}
\label{sec:preliminaries}

We refer to the vertex set of a given (possibly infinite) graph $G$ by $\V{G}$ and to the edge set by $\E{G}.$

A graph $G$ is called \emph{bipartite} if its vertex set can be partitioned into two independent sets $V_0$ and $V_1$, which we also refer to as the \emph{colour classes} of $G.$
When depicting bipartite graphs we depict elements of $V_1$ as filled circles and elements of $V_0$ as empty circles.

For a graph $G$ and a set $X\subseteq \V{G}$ we define the \emph{cut induced by $X$}, $\Cut{G}{X} \coloneqq \Set{e \in \E{G} \colon \Abs{e \cap X}=1} $, to be the set of all edges with exactly one endpoint in $X.$

\subsection{Connection to matching covered graphs}

Every strongly connected (possibly infinite) directed graph corresponds to a bipartite graph with a perfect matching.
We can obtain this bipartite graph with colour classes $V_1$ (black) and $V_0$ (white) from a digraph $D$ by splitting every vertex $v$ into two vertices $v_0,v_1$, one of each colour class, that are connected by an edge, such that all in-edges of $v$ become in-edges of the `white' vertex $v_0$ and all out-edges of $v$ become out-edges of the `black' vertex $v_1.$
Then we remove all directions from the edges.
Clearly, this yields a bipartite graph and the new edges form a perfect matching in it.
We refer to this graph as $\MatchingGraph{D}$ and to the canonical perfect matching in it as $\MatchingGraphMatching{D}.$

\begin{figure}[ht!]
    \centering
    \begin{tikzpicture}
        \node (directed) at (-3,0) {
            \begin{tikzpicture}
                \node (center) at (0,0) {};
                \foreach\i in {0,...,4}
                {
                    \node[vertex] (v-\i) at ($(center)+({360/5*\i}:2)$) {};
                }
                \draw[directededge,bend right] (v-0) to (v-1);
                \draw[directededge,bend right] (v-1) to (v-0);
                \draw[directededge,bend right] (v-2) to (v-3);
                \draw[directededge,bend right] (v-3) to (v-2);
                
                \draw[directededge,bend left] (v-1) to (v-2);
                \draw[directededge,bend left] (v-2) to (v-1);
                \draw[directededge,bend left] (v-3) to (v-4);
                \draw[directededge,bend left] (v-4) to (v-3);
                \draw[directededge,bend left] (v-0) to (v-4);
                \draw[directededge,bend left] (v-4) to (v-0);
            \end{tikzpicture}};
        
        \node (matching) at (3,0) {
            \begin{tikzpicture}
                \node (center) at (0,0) {};
                \foreach\i in {1,3,4}
                {
                    \node[vertexB] (v-o-\i) at ($(center)+({360/5*\i}:2.3)$) {};
                    \node[vertexW] (v-i-\i) at ($(center)+({360/5*\i}:1)$) {};
                }
                \foreach\i in {0,2}
                {
                    \node[vertexW] (v-o-\i) at ($(center)+({360/5*\i}:2.3)$) {};
                    \node[vertexB] (v-i-\i) at ($(center)+({360/5*\i}:1)$) {};
                }
                
                \draw[edge] (v-o-0) to (v-o-1);
                \draw[edge] (v-i-1) to (v-i-0);
                \draw[edge] (v-o-2) to (v-o-3);
                \draw[edge] (v-i-3) to (v-i-2);
                
                \draw[edge] (v-i-1) to (v-i-2);
                \draw[edge] (v-o-2) to (v-o-1);
                \draw[edge] (v-i-3) to (v-o-4);
                \draw[edge] (v-i-4) to (v-o-3);
                \draw[edge] (v-o-0) to (v-o-4);
                \draw[edge] (v-i-4) to (v-i-0);
                
                \foreach\i in {0,...,4}
                {
                    \draw[matching=myOrange](v-o-\i) to (v-i-\i);
                }
            \end{tikzpicture}};
    \end{tikzpicture}
    \caption{A strongly connected directed graph $D$ and the corresponding bipartite graph $\MatchingGraph{D}$ with the canonical perfect matching \textcolor{myOrange!80!red}{$\MatchingGraphMatching{D}$}.}
    \label{fig:directed_to_matching}
\end{figure}
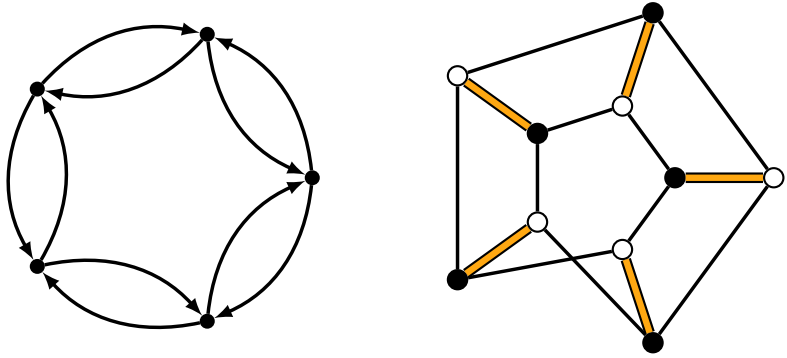

Conversely, we can also obtain a directed graph from any given (possibly infinite) bipartite graph with a perfect matching, by directing all non-matching edges from $V_1$ to $V_0$ and then contracting all matching-edges into single vertices.
For a given graph $G$ with perfect matching $M$, we call this the $M$-direction of $G$ and denote it $\DirM{G}{M}.$

For an undirected graph $G$ we refer to the set of all perfect matchings in $G$ by $\Perf{G}.$
An undirected graph $G$ is called \emph{matching covered with respect to a perfect matching $M \in \Perf{G}$}, if every edge lies in a perfect matching that has finite symmetric difference to the matching $M.$
We write $\Perf{G,M}$ for the set of all perfect matchings that have finite symmetric difference to $M.$
We simply say $G$ is \emph{matching covered}, when it is matching covered with respect to some matching.
Note that on finite graphs this is equivalent to saying that every edge lies in a perfect matching as $\Perf{G,M} = \Perf{G}$ for all perfect matchings $M$ in $G.$

\begin{observation}
    \label{obs:symmetric_difference}
    Given an undirected graph $G$ that is matching covered with respect to a perfect matching $M$, then the symmetric difference between any two perfect matchings in $\Perf{G,M}$ is a collection of disjoint cycles.
\end{observation}

\begin{definition}[tight cuts and tight sets, \BoBs]
    Let $G$ be a graph that is matching covered with respect to a perfect matching $M.$
    A subset $X \subseteq \V{G}$ is a \emph{tight set} in $G$ if every $M' \in \Perf{G,M}$ has exactly one edge in $\cut{}{X}$ and we refer to $\Cut{}{X}$ as a \emph{tight cut}.
    
    A tight set $X$ and its corresponding tight cut $\Cut{}{X}$ are called \emph{trivial} if $\Abs{X} = 1$ or $\Abs{\V{G}\setminus X} = 1.$

    Tight cuts $C_1, C_2$ are \emph{nested} if there are $X_1, X_2 \subset \V{G}$ with $X_1 \cap X_2 = \emptyset$ such that $C_i = \cut{}{X_i}$ for $i \in [2].$ Tight sets $X_1$ and $X_2$ are \emph{nested} if $\Cut{}{X_1}$ and $\Cut{}{X_2}$ are nested tight cuts.
    
    A bipartite graph with no non-trivial tight sets is called a \emph{brace}, while a non-bipartite graph with no non-trivial tight sets is called a \emph{brick}.
    Since thus every graph with no non-trivial tight sets is either a brick or a brace, we refer to such a graph as a \emph{\BoB} (Brick or Brace) for short.

    For a graph $G$ and a maximal family $\mathcal{C}$ of nested tight cuts in $G$, we define $\tightCutsToSets{\mathcal{C}} \coloneqq \Set{X \subseteq \V{G}: \cut{}{X} \in \mathcal{C}}$, i.e.~the set of tight sets corresponding to tight cuts in $\mathcal{C}.$
    Note that a tight set $X$ is contained in $\tightCutsToSets{\mathcal{C}}$ if $\cut{}{X}$ is nested with $\mathcal{C}$ since $\mathcal{C}$ is maximal.
    Furthermore, any two elements of $\tightCutsToSets{\mathcal{C}}$ are nested.
\end{definition}

In this paper we do not demand any additional properties of the perfect matching $M$ or the set $\Perf{G, M}$ of a matching covered graph $G.$
Therefore, unless stated otherwise, we consider tight cuts and tight sets with respect to an arbitrary but fixed perfect matching of $G$, without explicitly mentioning that perfect matching.

\begin{proposition} \label{prop:tight_cut_connected}
    Let $G$ be a connected, matching covered graph.
    For any tight set $X \subset \V{G}$ the subgraphs $G[X]$ and $G[\V{G} \setminus X]$ are connected.
\end{proposition}

\begin{proof}
For any two distinct edges $e$ and $f$ in $\cut{}{X}$, the symmetric difference of any matching containing $e$ with any matching containing $f$ is a disjoint union of cycles.
The cycle containing $e$ must also contain $f$, and so contains paths joining $e$ to $f$ through both $X$ and $V(G) \setminus X.$

Now for any two vertices $v$ and $w$ of $X$ we know there is a path $P$ joining them in $G$, and we can find a walk joining them in $G[X]$ by replacing each segment where $P$ leaves $G[X]$ with a path in $G[X]$ as above.
\end{proof}

\begin{definition}
    Let $D$ be a digraph.
    A tuple $\Brace{A,B}$ with $A,B \subseteq \V{D}$ and $A \cup B = \V{D}$ is a \emph{(directed) separation} if there is no edge with tail in $B\setminus A$ and head in $A\setminus B$ or there is no edge with tail in $A\setminus B$ and head in $B\setminus A.$
    The integer $k \coloneqq \Abs{A \cap B}$ is called the \emph{order} of the separation and we also refer to $\Brace{A,B}$ as a (directed) $k$-separation.
\end{definition}

\begin{proposition}
    The tight sets in a matching covered bipartite graph $G$ with respect to a perfect matching $M$ correspond one-to-one to the 1-separations in $\DirM{G}{M}.$
\end{proposition}
\begin{proof}
    Let $X$ be a tight set in $G.$
    Then $M$ has exactly one edge $e$ with one endpoint in $X$ and the other in $\V{G}\setminus X.$
    Let $X' \subseteq D \coloneqq \DirM{G}{M}$ be the set corresponding to the edges of $M$ lying in $X \setminus e$ and let $v_e$ be the vertex obtained by contracting the edge $e.$
    We show that $\Brace{X' \cup\Set{v_e},\V{D}\setminus X'}$ is a 1-separation.
    
    Without loss of generality assume the unique vertex in $e \cap X$ is in $V_1.$
    Then there is no edge between a vertex in $X\cap V_0$ and $\V{G}\setminus X$, for the following reason.
    Suppose there was such an edge, then, as $G$ is matching covered, there is a matching $M' \in \Perf{G,M}$ containing it.
    We construct a path in $X$ alternating between the two matchings $M$ and $M'$ starting with this edge.
    Then every edge from $M'$ ends in a vertex of $V_0$ and thus every edge of $M$ we add starts in a vertex of $V_0.$
    Therefore we never add $e$, and never close a cycle.
    But this implies that the symmetric difference between $M$ and $M'$ is infinite, a contradiction.
    Thus all edges between $X$ and $\V{G}\setminus X$ have their endpoint in $X$ within the colour class $V_1.$
    Therefore there are no edges from $\Brace{\V{D}\setminus X'}\setminus\Brace{X' \cup\Set{v_e}}$ to $\Brace{X' \cup\Set{v_e}}\setminus\Brace{\V{D}\setminus X'}.$
    
    Let $\Brace{A,B}$ be a 1-separation in $D$ with separation vertex $v$ such that there are no edges from $B\setminus A$ to $A \setminus B.$
    Let $v_0 \in V_0$ and $v_1 \in V_1$ be the two vertices such that the edge $\Brace{v_0,v_1} \in M$ gets contracted to $v.$
    
    Consider the set $X$ obtained by taking all vertices in $G$ that are contracted to a vertex in $A\setminus B$ together with the vertex $v_0.$
    We claim that $X$ is tight.
    Suppose towards a contradiction that it is not.
    First, we consider the case that there is a matching $M' \in \Perf{G,M}$ having no edge with one endpoint in $X$ and one in $\V{G}\setminus X.$
    Then consider the component $C$ of the symmetric difference between $M$ and $M'.$
    As $v_0$ has no neighbour in $X$ it is matched to by $M$ and $M'$ has no edge leaving $X$, $C$ cannot be a cycle, this contradicts, by \cref{obs:symmetric_difference}, that $M' \in \Perf{G,M}.$
    
    Second, we consider the case that there is a matching $M' \in \Perf{G,M}$ with more than one edge having exactly one endpoint in $X.$
    For every edge of $M'\setminus \Set{\Brace{v_0,v_1}}$ having exactly one endpoint in $X$ consider the component of the symmetric difference between $M$ and $M'$ containing it.
    As at most one of the edges can be contained in a cycle with $\Brace{v_0,v_1}$, there is at least one that is contained in an infinite path, contradicting that $M' \in \Perf{G,M}.$
\end{proof}

\subsection{Tight set partitions}
Given two partitions $\mathcal{P}, \mathcal{P}'$ of some set, we say \emph{$\mathcal{P}$ refines $\mathcal{P}'$} if every partition class of $\mathcal{P}$ is subset of a partition class of $\mathcal{P}'.$
\begin{definition}
    A partition $\mathcal{P}$ of the vertex set of a matching covered graph $G$ is a \emph{tight set partition} of $G$ if every $P \in \mathcal{P}$ is a tight set.
    
    For every tight set partition $\Partition$ of a matching covered graph $G$ we define the collapse $\collapse{\Partition}$ of $\Partition$ to be the graph with vertex set $\Partition$ and an edge between $P$ and $Q$ if and only if there are $p \in P$ and $q \in Q$ such that $pq \in \E{G}.$
\end{definition}

\begin{lemma}
    \label{lem:collapse_matching_covered}
    Let $\mathcal{P}$ be a tight set partition of a matching covered graph $G$, then $\collapse{\mathcal{P}}$ is matching covered.
\end{lemma}
\begin{proof}
    Let $M$ be the perfect matching of $G$ with respect to which $G$ is matching covered.
    For every perfect matching $M' \in\Perf{G,M}$ we define
    \begin{align*}
        \Fkt{M'}{\mathcal{P}} \coloneqq \Set{\Set{P_i,P_j} \colon \text{there exist $x_i\in P_i$ and $x_j \in P_j $ with } \Set{x_i,x_j}\in M'}.
    \end{align*}
    We first show that this yields a perfect matching in the collapse.
    \begin{claim}
        \label{claim:perfect_matching}
        For every $M' \in\Perf{G,M}$ the obtained set of edges $\Fkt{M'}{\mathcal{P}}$ is a perfect matching of $\collapse{\mathcal{P}}.$
    \end{claim}
    \begin{claimproof}
        As $P \in \mathcal{P}$ is tight $\Abs{M' \cap \Cut{G}{P}} = 1.$
        Thus, every $P \in \V{\collapse{\mathcal{P}}}$ lies in exactly one edge of $\Fkt{M'}{\mathcal{P}}.$
    \end{claimproof}
    
    Next we show that the collapse is matching covered.
    \begin{claim}
        \label{claim:collapse}
            The collapse $\collapse{\mathcal{P}}$ is matching covered with respect to $\Fkt{M}{\mathcal{P}}.$
    \end{claim}
    \begin{claimproof}
        Let $\Set{P_i,P_j} \in\E{\collapse{\mathcal{P}}}.$
        By definition of collapses, there are vertices $x_i \in P_i$ and $x_j \in P_j$ with $\Set{x_i,x_j} \in \E{G}.$
        So there is a matching $M' \in\Perf{G,M}$ with $\Set{x_i,x_j}\in M'.$
        Because the symmetric difference between $M$ and $M'$ is finite, so is the symmetric difference between $\Fkt{M'}{\mathcal{P}}$ and $\Fkt{M}{\mathcal{P}}.$
        Thus we obtain $\Fkt{M'}{\mathcal{P}} \in \Perf{\collapse{\mathcal{P}},\Fkt{M}{\mathcal{P}}}.$\qedhere
    \end{claimproof}
    Together \cref{claim:perfect_matching} and \cref{claim:collapse} imply the statement.
\end{proof}

\begin{lemma}
    \label{lem:tight_sets_in_collapse}
    Let $\mathcal{P}$ be a tight set partition of a matching covered graph $G$ and let $X$ be a subset of $\mathcal{P}.$
    Then $X$ is tight in $\collapse{\mathcal{P}}$ if and only if $\bigcup X$ is tight in $G.$
\end{lemma}
\begin{proof}
    Let $M$ be the perfect matching of $G$ with respect to which $G$ is matching covered. By \cref{lem:collapse_matching_covered}, this implies that $\collapse{\mathcal{P}}$ is matching covered with respect to the matching $\Fkt{M}{\mathcal{P}}.$
    
    Suppose first of all that $X$ is tight in $\collapse{\mathcal{P}}.$
    Then for any $M' \in\Perf{G,M}$ the matching $\Fkt{M'}{\mathcal{P}}$ has exactly one edge in $\Cut{\collapse{\mathcal{P}}}{X}.$
    Thus, there is exactly one set $P \in X$ and exactly one $Q \in \V{\collapse{\mathcal{P}}}\setminus X$ with $PQ \in \E{\collapse{\mathcal{P}}} \cap \Fkt{M'}{\mathcal{P}}.$
    Thus there are vertices $x\in P$ and $y\in Q$ with $xy\in \E{G}\cap M'.$
    As $P$ and $Q$ are tight sets, there is exactly one such edge, thus $\Abs{M' \cap \Cut{G}{\bigcup X}} = 1.$
    
    Conversely, suppose that $\bigcup X$ is tight in $G$, and consider any $M' \in \Perf{\collapse{\mathcal{P}}, \Fkt{M}{\mathcal{P}}}.$ Let the finitely many edges in $M' \setminus \Fkt{M}{\mathcal{P}}$ be $P_1Q_1$, $P_2Q_2$, \ldots $P_nQ_n.$ For $i \leq n$ let $e_i$ be any edge of $G$ with one endpoint in $P_i$ and the other in $Q_i$, let $M_i$ be any element of $\Perf{G,M}$ containing $e_i$ and let $N_i$ be the set of edges in $M_i$ with both endpoints in $P_i \cup Q_i.$ Let $N \coloneqq \bigcup_{i = 1}^n N_i.$ Finally, let $M''$ be obtained from $M$ by removing all edges with an endpoint in any $P_i$ or $Q_i$ and adding the edges in $N.$ 
    
    By construction $M''$ is a perfect matching of $G$ whose symmetric difference with $M$ is a subset of $\bigcup_{i = 1}^n(M \triangle M_i)$ and so is finite. So $M'' \in \Perf{G,M}$, and so there is a unique edge in $M'' \cap \cut{}{\bigcup X}.$ Since by construction also $\Fkt{M''}{\mathcal{P}} = M'$, this implies that there is a unique edge in $M' \cap \cut{}{X}.$ Since $M'$ was arbitrary, this implies that $X$ is tight in $\collapse{\mathcal{P}}$, as required.
\end{proof}

\section{Even and odd sets in infinite graphs}
\label{sec:even_odd}

Let $G$ be a matching covered graph with respect to a perfect matching $M.$
A set $X \subseteq \V{G}$ has a \emph{parity} if $\Abs{\Cut{G}{X} \cap M}$ is finite.
For a set $X$ with a parity such that $\Abs{\Cut{G}{X} \cap M}$ is even, we say that $X$ is \emph{even} and if $\Abs{\Cut{G}{X} \cap M}$ is odd, we call $X$ \emph{odd}.

\begin{lemma}
    \label{lem:even_odd_same_for_all_matchings}
    Let $G$ be a matching covered graph with respect to a perfect matching $M.$
    If $X \subseteq \V{G}$ is odd, then every perfect matching $M' \in \Perf{G,M}$ contains a finite, odd number of edges with exactly one endpoint in $X.$
    Similarly, if $X \subseteq \V{G}$ is even, then every perfect matching $M' \in \Perf{G,M}$ contains a finite, even number of edges with exactly one endpoint in $X.$
\end{lemma}
\begin{proof}
    Let $M' \in \Perf{G,M}$ and $X\subseteq \V{G}.$
    As the symmetric difference between $M$ and $M'$ is finite, $M'$ has finitely many edges with exactly one endpoint in $X$ if and only if $M$ does.
    Consider the case that both have finitely many edges with exactly one endpoint in $X.$
        
    We split the set of edges $\Cut{}{X}\cap \Brace{M \cup M'}$ in two disjoint parts:
        the ones lying in $M$-$M'$-alternating cycles and the ones lying in $M \cap M'.$
    As every cycle has an even number of edges in $\Cut{}{X}$, it either contains an even number of edges from both matchings or an odd number of edges from both matchings.
    Thus, $M$ and $M'$ have the same parity of edges lying in $M$-$M'$-alternating cycles.
    Clearly, the parity of edges from both matchings lying in $M \cap M'$ is also the same.
    Thus, the intersections $\Cut{}{X} \cap M$ and $\Cut{}{X} \cap M'$ have the same parity.
\end{proof}

In the following \namecref{lem:even_odd} we make some simple observations about even and odd sets and show that these definitions behave intuitively.
We use these properties throughout the paper, often without explicit reference to this \namecref{lem:even_odd}.

\begin{lemma}
\label{lem:even_odd}
    Let $G$ be a matching covered graph with respect to a perfect matching $M$ and $X,Y$ two subsets of $\V{G}.$
    \renewcommand{\labelenumi}{\textbf{\theenumi}}
    \renewcommand{\theenumi}{(P.\arabic{enumi})}
    \begin{enumerate}[labelindent=0pt,labelwidth=\widthof{\ref{last-item-even_odd}},itemindent=1em]
        \item \label{even-odd:odd_and_odd_is_even} if $X$ and $Y$ are disjoint and both odd, then $X \cup Y$ is even,
        \item \label{even-odd:odd_and_even_is_odd} if $X$ and $Y$ are disjoint and $X$ is even while $Y$ is odd, then $X \cup Y$ is odd,
            \label{last-item-even_odd}
        \item\label{even-odd:even_and_even_is_even} if $X$ and $Y$ are disjoint and both even, then $X \cup Y$ is even,
        \item\label{even-odd:intersections_are_even_or_odd} if $X$ and $Y$ both have a parity, then $X \cap Y$ has a parity, and
        \item\label{even-odd:differences_are_even_or_odd} if $X$ and $Y$ both have a parity, then $X \setminus Y$ has a parity.
    \end{enumerate}
\end{lemma}
\begin{proof}
    We prove the points separately.
    
    \paragraph{\Cref{even-odd:odd_and_odd_is_even}}
        Assume that $X$ and $Y$ are disjoint and both odd.
        We can partition the edges in $\Cut{}{X \cup Y} \cap M$ into three parts:
            edges with one endpoint in $X$ and the other in $Y$, remaining edges with one endpoint in $X$ and remaining edges with one endpoint in $Y.$
            
        If there is an even number of edges in $M$ with one endpoint in $X$ and the other in $Y$, then the remaining two sets are both odd and thus $\Cut{}{X \cup Y} \cap M$ is even.
        If there is an odd number of edges in $M$ with one endpoint in $X$ and the other in $Y$, then the remaining two sets are both even and thus $\Cut{}{X \cup Y} \cap M$ is even.
        
    \paragraph{\Cref{even-odd:odd_and_even_is_odd}, \cref{even-odd:even_and_even_is_even}} Can be proved with arguments akin to \cref{even-odd:odd_and_odd_is_even}.
        
    \paragraph{\Cref{even-odd:intersections_are_even_or_odd}}
        Consider any $e \in M$ with exactly one endpoint in $X \cap Y.$
        If $e \subset X$, then $e$ has exactly one endpoint in $Y$, if $e \subset Y$, then $e$ has exactly one endpoint in $X$ and if $e \not\subset X$ and $e \not \subset Y$, then its other endpoint is in $\V{G} \subseteq (X \cup Y).$
        Thus, in every case $e$ lies in at least one of $M \cap \Cut{}{X}$ or $M \cap \Cut{}{Y}$ and therefore, there are only finitely many such edges.
    
    \paragraph{\Cref{even-odd:differences_are_even_or_odd}}
        Consider any $e \in M$ with exactly one endpoint in $X \setminus Y.$
        If $e \subset X$, then $e$ has exactly one endpoint in $Y$, thus $e$ lies in $M \cap \Cut{}{Y}.$
        If $e \not \subset X$, then $e$ lies in $M \cap \Cut{}{X}.$
        Therefore there are only finitely many such edges.
\end{proof}

\section{Crossing tight sets and tight set partitions}
\label{sec:crossings}

For the notion of torsoids (see \cref{sec:torsoids}) we have to understand how tight sets interact with each other.
We begin with the investigation of crossing tight sets in \cref{subsec:Basic_facts_about_tight_sets} and based on this, we define and study passable sets in \cref{subsec:passable_sets}.
Thereafter, we explore how tight set partitions interact with tight sets in \cref{subsec:interaction_tight_set_tight_set_partition} and introduce
a relation of tight set partitions that we call \emph{correspondence} in \cref{subsec:correspondence_tight_set_partitions}.

\subsection{Basic facts about crossing tight sets}\label{subsec:Basic_facts_about_tight_sets}
For this \namecref{subsec:Basic_facts_about_tight_sets} we fix a matching covered graph $G.$
The first few results in this \namecref{subsec:Basic_facts_about_tight_sets} were already known for finite matching covered graphs \cite{Edmonds1982,matchingtheory_book,Lovasz1987}, and their proofs for infinite matching covered graphs do not contain any new ideas.
However, to keep this paper self-contained we give these proofs here.

\begin{lemma}\label{lem:basictight}
    Let $X$ and $X'$ be tight sets with $X \cap X'$ odd.
    Then both $X \cap X'$ and $X \cup X'$ are tight sets and there is no edge with endpoints in both $X \setminus X'$ and $X' \setminus X.$
    If $X \setminus X' \neq \emptyset$ then there is an edge with endpoints in $X \cap X'$ and $X \setminus X'.$ 
\end{lemma}
\begin{proof}
    First we show that $X \cap X'$ is tight.
    For any matching $M$, \[|M \cap \partial(X \cap X')| \leq |M \cap \partial(X)| + |M \cap \partial(X')| = 2\,,\]
    but since $X \cap X'$ is odd, $|M \cap \partial(X \cap X')|$ must also be odd and so must be equal to one.
    An identical argument shows that $X \cup X'$ is tight.
    
    Next, suppose for a contradiction that there is an edge $e$ with endpoints in $X \setminus X'$ and $X' \setminus X.$
    Let $M$ be any matching containing $e$, and let $f$ be the unique edge in $M \cap \partial(X \cap X').$
    As $f \in \partial(X)$ and $\partial(X)$ contains both $e$ and $f$, contradicting tightness of $X.$
    Thus there can be no such an edge $e.$
    
    Finally, suppose that $X \setminus X'$ is nonempty.
    By connectivity of $G$ there must be some edge $e \in \partial(X \setminus X').$
    Let $M$ be a matching containing $e.$
    Then $|M \cap \partial(X \setminus X')|$ is a positive even number, so it is at least 2.
    Every edge in $M \cap \partial(X \setminus X')$ must be in one of $M \cap \partial(X)$ and $M \cap \partial(X')$, and each of those sets contains only one edge, so one of the edges of $M \cap \partial(X \setminus X')$ must be in $M \cap \partial(X').$
    Then that edge has one endpoint in $X \setminus X'$ and the other in $X'.$
    We have just seen that the endpoint in $X'$ cannot be in $X' \setminus X$, and so it must be in $X \cap X'.$
\end{proof}

\begin{lemma}\label{lem:basicnoedge}
    Let $X$, $X'$ and $X''$ be tight sets such that $X \cap X'$ and $X \cap X''$ are even and disjoint.
    Then there is no edge between these two sets.
\end{lemma}
\begin{proof}
    Applying \cref{lem:basictight} to $X$ and the complement of $X'$, we see that $X \setminus X'$ is tight.
    But then applying the same \namecref{lem:basictight} to $X \setminus X'$ and the complement of $X''$ we get the desired result.
\end{proof}

\begin{lemma}\label{lem:nothree}
    Let $X$, $X'$ and $X''$ be tight sets such that $X \cap X' = X \cap X'' = X' \cap X''$ is even.
    Then $X \cap X' = X \cap X'' = X' \cap X''$ is empty.
\end{lemma}
\begin{proof}
    Suppose not for a contradiction.
    By \cref{lem:basictight} applied to $X$ and the complement of $X'$ there is an edge $e$ with one end in $X \setminus X'$ and the other in $X \cap X' = X' \cap X''.$
    But this contradicts \cref{lem:basictight} applied to $X'$ and the complement of $X''.$
\end{proof}

\begin{lemma}\label{lem:infunion}
    Let $(X_i)_{i \in I}$ be a family of tight sets such that the union of any two of them is also tight.
    Then $X \coloneqq \bigcup_{i \in I} X_i$ is either also tight or is the whole vertex set.
\end{lemma}
\begin{proof}
    Suppose for a contradiction that there is a matching $M$ containing two edges $e$ and $f$ in the boundary of $X.$
    Choose $i,j \in I$ such that the endpoints of $e$ and $f$ in $X$ are in $X_i$ and $X_j$ respectively.
    Then both $e$ and $f$ are also in the boundary of $X_i \cup X_j$, contradicting the tightness of this set.
    
    Thus any matching $M$ contains at most one edge in $\cut{}{X}.$ If $X$ is not the whole vertex set then there is a matching containing some edge, and so exactly one edge, in the boundary of $X.$ Thus $X$ is odd, and any matching contains exactly one edge in the boundary of $X$ as required.
\end{proof}

Applying this to the complements of a family of tight sets also gives the following statement.

\begin{corollary}\label{cor:infintersection}
    Let $(X_i)_{i \in I}$ be a family of tight sets such that the intersection of any two of them is odd.
    Then $X \coloneqq \bigcap_{i \in I} X_i$ is either also tight or else is empty.\hfill$\square$
\end{corollary}

\begin{lemma}\label{lem:deletemany}
Let $X$ be a tight set and $(X_i)_{i 
\in I}$ a family of tight sets whose intersections with $X$ are even and disjoint.
Then $X \setminus \bigcup_{i \in I}X_i$ is tight.
\end{lemma}
\begin{proof}
    By repeated applications of \cref{lem:basictight}, $(X \setminus X_i)_{i \in I}$ is a family of tight sets and any intersection of two of its elements is tight.
    So by applying \cref{cor:infintersection} to this family we see that $X \setminus \bigcup_{i \in I}X_i$ is either tight or empty.
    
    Suppose for a contradiction that it is empty.
    Let $M$ be a perfect matching, let $e$ be the edge of $M$ in the boundary of $X$ and let $i \in I$ be chosen such that the endpoint of $e$ in $X$ is in $X_i.$
    Since $X \cap X_i$ is even there must be some other edge $f \in M$ in the boundary of $X_i$, and the endpoint of $f$ outside $X_i$ must lie in some $X_j$ with $j \neq i.$
    But this contradicts \cref{lem:basicnoedge}.
\end{proof}

\subsection{Characterising cycles}

\begin{lemma}\label{lem:charcycle}
    Let $G$ be a matching covered graph with at least six vertices and with a cyclic order on its vertex set such that any set of three consecutive vertices is tight.
    Then $G$ is a cycle.
\end{lemma}
\begin{proof}
    Suppose for a contradiction that there is an edge $e$ joining vertices $v$ and $w$ which are not adjacent in the cyclic order.
    Let $M$ be a matching containing $e.$
    Let $x$ and $y$ be the neighbours of $v$ in the cyclic order and let $x'$ and $y'$ be the neighbours of $x$ and $y$ in the cyclic order other than $v.$
    Since $G$ has at least $6$ vertices, $x' \neq y'$ and so without loss of generality $w \neq x'.$
    Since $e$ is in the boundary of $\Set{x,v,y}$, no other edge of $M$ can be, so the edge in $M$ incident to $x$ can only be $xy.$
    But then both $e$ and $xy$ are in the boundary of the tight set $\Set{x',x,v}$, giving the desired contradiction.
\end{proof}
For a cycle $C$ we call a nonempty proper subset $I$ of $V(C)$ an \emph{interval} if $C[I]$ is connected.
\begin{remark} \label{rem:cycle_tight_cuts}
    Let $C$ be a matching covered cycle.
    A set $X \subseteq V(C)$ is tight in $C$ if and only if it is an interval in $C$ with odd cardinality.
\end{remark}
\begin{lemma}\label{lem:extendcycle}
    Let $G$ be a matching covered graph and let $X = \Set{p_1, p_2, p_3}$ be a tight set such that collapsing $X$ in $G$\footnote{this means that we take the collapse of the tight set partition consisting of $X$ and the singletons of all vertices not in $X.$} gives rise to a graph $H$ whose vertices can be ordered cyclically such that any set of three consecutive vertices is tight.
    Let the neighbours of $X$ in $H$ be $q$ and $q'.$
    Suppose further that both $\Set{q, p_1, p_2}$ and $\Set{p_2, p_3, q'}$ are tight.
    Then $G$ is a cycle. 
\end{lemma}
\begin{proof}
    Let $r$ and $r'$ be the neighbours of $q$ and $q'$ other than $X$ after collapsing $X.$
    Note that since collapsing $X$ in $G$ gives rise to a cycle, $G$ has at least six vertices and $r,r^\prime \not \in \Set{p_1,p_2,p_3,q,q^\prime}.$
    By \cref{lem:charcycle} together with \cref{rem:cycle_tight_cuts} it suffices to show that $\Set{r, q, p_1}$ and $\Set{p_3, q', r'}$ are tight, and by symmetry it suffices to show that the first of these is.
    But this follows by applying \cref{lem:basictight} to $\Set{r, q, p_1, p_2, p_3}$ and the complement of $\Set{p_2, p_3, q'}.$ 
\end{proof}

\subsection{Passable sets}
\label{subsec:passable_sets}

We now need a notion capturing when we can make small modifications to tight set partitions by moving some elements from one partition class to another.
As in \cref{subsec:Basic_facts_about_tight_sets} we fix a matching covered graph $G$ for this \namecref{subsec:passable_sets}.

\begin{definition}[passable]
    A set $S$ is \emph{passable} for $P$ if $P \setminus S$  and $P \cup S$ are tight.
    It is \emph{passable between $P$ and $Q$} if $S \subseteq P \cup Q$ and $S$ is passable for both.
\end{definition}

\begin{remark}
    Note that any passable set is even, since both $P\setminus S$ and $P \cup S$ are tight and therefore odd.
\end{remark}

\begin{lemma}\label{lem:notodd}
  Let $S$ and $S'$ be sets which are passable between disjoint tight sets $P$ and $Q$ whose union is not the whole vertex set.
  Then $P \cap S \cap S'$ is even. 
\end{lemma}
\begin{proof}
    Suppose not for a contradiction.
    We may replace $S$ with $S \cap P$ if necessary, since it is passable between $P$ and $Q$ as well.
    Thus we may assume without loss of generality that $S$ is a subset of $P.$
    Similarly we may assume that $S'$ is also a subset of $P.$
    By assumption $P \cap S \cap S' = S \cap S'$ is odd, which implies that $S \setminus S'$ and $S' \setminus S$ are also odd, since $S$ and $S'$ are passable.
    
    The set $P \setminus S$ is tight since $S$ is passable for $P.$
    Let $A$ be $\V{G} \setminus (P \cup Q).$
    Now applying \cref{lem:basictight} again to $P \setminus S$ and $Q \cup S'$ we see that their union $(P \cup Q) \setminus (S \setminus S')$ is tight, hence so is its complement $A \cup (S \setminus S').$
    Similarly $A \cup (S' \setminus S)$ is tight.
    Finally, $A \cup Q$ is tight since it is the complement of $P.$
    Applying \cref{lem:nothree} to these three tight sets yields the desired contradiction, since by assumption $A$ is nonempty.
\end{proof}

\begin{lemma}\label{lem:passcup}
    Let $S$ and $S'$ be sets which are passable between disjoint tight sets $P$ and $Q$ whose union is not the whole vertex set.
    Then $S \cup S'$ is also passable between $P$ and $Q.$
\end{lemma}
\begin{proof}
    By symmetry it suffices to show that $S \cup S'$ is passable for $P.$
    We know that $P \cap S$ and $P \cap S'$ are even since $P \setminus S$ and $P \setminus S'$ are odd, and $P \cap S \cap S'$ is even by \cref{lem:notodd}.
    $P \cap (S \cup S')$ is even and therefore $P \setminus (S \cup S')$ is odd.
    So applying \cref{lem:basictight} to $P \setminus S$ and $P \setminus S'$ shows that $P \setminus (S \cup S')$ is tight.
    
    A symmetric argument shows that $Q \cap (S \cup S')$ is also even and so $P \cup (S \cap S')$ is odd.
    Thus applying \cref{lem:basictight} to $P \cup S$ and $P \cup S'$ shows that $P \cup S \cup S'$ is tight.
    Thus $S \cup S'$ is passable for $P$, as required.
\end{proof}

\begin{lemma}\label{lem:maxpass}
    Let $P$ and $Q$ be disjoint tight sets whose union is not the whole vertex set. Let $\mathcal{S}$ be any set of sets which are passable between $P$ and $Q.$ 
    Then $S \coloneqq \bigcup \mathcal{S}$ is itself passable between $P$ and $Q.$ In particular, if $\mathcal{S}$ is the set of all sets which are passable between $P$ and $Q$ then $S$ is the largest such set.
\end{lemma}
\begin{proof}
    By symmetry it suffices to show that $S$ is passable for $P.$  
    $(P \cup T)_{T \in \mathcal{S}}$ is a family of tight sets and by \cref{lem:passcup} the union of any two of them is also tight, so by \cref{lem:infunion} the union of the whole family is tight.
    A similar argument shows that $Q \cup S$ is also tight, and so $P \setminus S = P \setminus (Q \cup S)$ is tight by \cref{lem:basictight} applied to the complement of $P$ and to $Q \cup S.$
\end{proof}

\begin{lemma}\label{lem:modifypass}
Let $P$ be a tight set, let $S$ be a passable set for $P$ and let $(S_i)_{i \in I}$ be a family of passable sets for $P$ such that $S$ and all of the $S_i$ are disjoint and such that $P \cup S \cup \bigcup_{i \in I}S_i$ is not the whole vertex set. Suppose further that all the sets of the form $(P \cup S) \setminus S_i$ or $(P \cup S_i) \setminus S$ are tight.
Then $S$ is passable for both $P \cup \bigcup_{i \in I}S_i$ and $P \setminus \bigcup_{i \in I}S_i.$
\end{lemma}
\begin{proof}
    The set $(P \cup \bigcup_{i \in I} S_i) \cup S$ is tight by \cref{lem:infunion} applied to the sets $X_i \coloneqq (P \cup S_i) \cup S$, which are all tight by \cref{lem:basictight}. Similarly, the set $(P \cup \bigcup_{i \in I}S_i) \setminus S$ is tight by \cref{lem:infunion} applied to the sets $X_i \coloneqq (P \cup S_i) \setminus S.$ This shows that $S$ is passable for $P \cup \bigcup_{i \in I}S_i.$
    
    The set $(P \setminus \bigcup_{i \in I} S_i) \cup S$ is tight by \cref{lem:deletemany} applied to $P \cup S$ and to the $X_i$ given by the complements of the tight sets $(P \cup S) \setminus S_i.$ Similarly $(P \setminus \bigcup_{i \in I} S_i) \setminus S$ is tight by \cref{lem:deletemany} applied to $P \setminus S$ with the same choice of $X_i.$ This shows that $S$ is passable for $P \setminus \bigcup_{i \in I} S_i.$
\end{proof}

\subsection{The interaction between tight sets and tight set partitions}
\label{subsec:interaction_tight_set_tight_set_partition}

\begin{definition}
    Let $G$ be a matching covered graph.
    Let $\mathcal{P}$ be a tight set partition of $G$ and let $X$ be a tight set in $G.$
    We denote by $\oddIntersections{\mathcal{P}}{X}$ the set of all elements of $\mathcal{P}$ whose intersection with $X$ is odd.
\end{definition}
For the rest of this \namecref{subsec:interaction_tight_set_tight_set_partition}, fix a matching covered graph $G.$
\begin{lemma}\label{lem:uncrosscut}
    Let $\mathcal{P}$ be a tight set partition of $G$ and let $X$ be a tight set.
    Then $\bigcup \oddIntersections{\mathcal{P}}{X}$ is a tight set. Furthermore, $\oddIntersections{\mathcal{P}}{X} \notin \Set{\emptyset, \mathcal{P}}.$
\end{lemma}
\begin{proof}
    Let $Y$ be the complement of $X.$
    Then by \cref{lem:deletemany} the set $Y \setminus \bigcup \oddIntersections{\mathcal{P}}{X}$ is tight, and therefore so is its complement $X \cup \bigcup \oddIntersections{\mathcal{P}}{X}.$ 
    Applying \cref{lem:deletemany} again, we get that \[\left(X \cup \bigcup \oddIntersections{\mathcal{P}}{X}\right) \setminus \bigcup (\mathcal{P} \setminus \oddIntersections{\mathcal{P}}{X})\] is tight, but this set is just $\bigcup \oddIntersections{\mathcal{P}}{X}$, giving the desired result.

    Since tight sets are neither empty nor the whole vertex set, and $\bigcup \oddIntersections{\Partition}{X}$ is one, $\oddIntersections{\mathcal{P}}{X}$ is neither empty nor $\mathcal{P}.$
\end{proof}

\begin{lemma}\label{lem:cutswell}
    Let $\Partition$ be a tight set partition of $G$ such that $\collapse{\Partition}$ is a \BoB, and let $X$ be a tight set.
    Then $\oddIntersections{\Partition}{X}$ is either a singleton or the complement of $\oddIntersections{\Partition}{X}$ in $\Partition$ is a singleton.
\end{lemma}
\begin{proof}
    By \cref{lem:uncrosscut} the set $\bigcup \oddIntersections{\Partition}{X}$ is tight, meaning that $\oddIntersections{\Partition}{X}$ is a tight set in $\collapse{\Partition}$ by \cref{lem:tight_sets_in_collapse}.
    Since $\collapse{\Partition}$ is a \BoB, the result follows.
\end{proof}

\begin{lemma}\label{lem:cutscyclicpartition}
    Let $\mathcal{P}$ be a tight set partition of $G$ such that $\collapse{\mathcal{P}}$ is a cycle, and let $X$ be a tight set in $G.$
    Let $n\coloneqq |\mathcal{P}|$  and $m\coloneqq |\oddIntersections{\mathcal{P}}{X}|.$
    
    Then there is a cyclic enumeration $P_1, \dots, P_n$ of the vertices in $\collapse{\mathcal{P}}$ such that $P_i \in \oddIntersections{\mathcal{P}}{X}$ if and only if $i \in [m].$
    Furthermore,
    \begin{equation*}
        \bigcup_{1 < i < m} P_i \quad \subset \quad X \quad \subset \quad \V{G} \setminus \bigcup_{ m + 1 < i < n } P_i
    \end{equation*}
    holds.
    If $3 \leq m \leq n - 3$ holds, then $P_1 \setminus X, P_n \cap X$ are passable between $P_1$ and $P_n$, and $P_{m} \setminus X, P_{m+1} \cap X$ are passable between $P_{m}$ and $P_{m + 1}.$
\end{lemma}
\begin{proof}
    The set $\bigcup \oddIntersections{\mathcal{P}}{X}$ is tight by \cref{lem:uncrosscut} and thus $\oddIntersections{\mathcal{P}}{X}$ is tight in $\collapse{\mathcal{P}}$ by \cref{lem:tight_sets_in_collapse}.
    By \cref{rem:cycle_tight_cuts}, $\oddIntersections{\mathcal{P}}{X}$ is an interval in $\collapse{\mathcal{P}}.$
    Thus we can choose the desired enumeration of the vertices in $\collapse{\mathcal{P}}.$

    Let $1 <  i < m$ be arbitrary.
    We prove that $P_i \subset X.$
    Suppose for a contradiction that $P_i \setminus X \neq \emptyset.$
    As $P_i \in \oddIntersections{\mathcal{P}}{X}$, $P_i \cap X$ is odd, which implies that $\Brace{\V{G} \setminus X} \cap \Brace{\V{G} \setminus P_i}$ is odd.
    Applying \cref{lem:basictight} to $\V{G} \setminus X$ and $\V{G} \setminus P_i$ shows that there is an edge $e$ with endpoints in $ (\V{G}\setminus X) \setminus (\V{G}\setminus P_i) = P_i \setminus X$ and $(\V{G} \setminus X ) \cap (\V{G} \setminus P_i) = \V{G} \setminus \Brace{X \cup P_i}.$
    As $\collapse{\mathcal{P}}$ is cyclic, there is $j \in \Set{i - 1, i + 1}$ such that $e$ has an endpoint in $P_j.$
    The set $X'\coloneqq X \cup P_j$ is a tight set by \cref{lem:basictight}, as $P_j \in \oddIntersections{\mathcal{P}}{X}.$
    The edge $e$ is an edge with endpoints in $P_i \setminus X'$ and $X' \setminus P_i.$
    This contradicts \cref{lem:basictight}, as $P_i \cap X' = P_i \cap X$ is odd.
    Therefore $P_i \subset X.$
    Applying the same argument to the complement of $X$ shows that $P_i \cap X = \emptyset$ for all $m + 1 < i < n.$
    Thus the desired subset relation is true.

    From now on we suppose that $3 \leq m \leq n - 3$ holds.
    We show that $P_n \cap X$ is passable between $P_1$ and $P_n.$
    By considering the complement of $X$ and/or reversing the enumeration one can show that the other sets are passable as well.
    $P_n \cup (X \cap P_n) = P_n$ and $P_1 \setminus (X \cap P_n) = P_1$ are tight by definition.
    $P_n \cap X$ is even since $P_n \notin \oddIntersections{\Partition}{X}$, thus $P_n \setminus X = (\V{G} \setminus X) \cap P_n $ is odd and thus tight by \cref{lem:basictight}.
    To see that $P_1 \cup (X \cap P_n)$ is tight we first note that by \cref{rem:cycle_tight_cuts}, the set $P_1 \cup P_n \cup P_{n -1}$ is tight.
    Consider the set $X' \coloneqq X \cup P_1$, which is tight by \cref{lem:basictight} since $P_1 \in \oddIntersections{\Partition}{X}.$
    By assumption, $m < n-1$ and thus $X' \cap \Brace{P_1 \cup P_n \cup P_{n -1}} = P_1 \cup \Brace{X \cap P_n}$ is odd and therefore tight.
    This concludes the proof that $X \cap P_n$ is passable between $P_1$ and $P_n.$
\end{proof}

\subsection{Correspondences between tight set partitions}
\label{subsec:correspondence_tight_set_partitions}

Throughout this subsection we work with a fixed matching covered graph $G.$

\begin{definition}
    Let $\mathcal{P}$ and $\mathcal{Q}$ be tight set partitions of $G.$
    A \emph{correspondence between $\mathcal{P}$ and $\mathcal{Q}$} is a bijection $\rho \colon \mathcal{P} \to \mathcal{Q}$ such that for any $P \in \mathcal{P}$ and $Q \in \mathcal{Q}$ we have that $P \cap Q$ is odd if and only if $Q = \rho(P).$
    If there is such a $\rho$ then we say that $\mathcal{P}$ \emph{is in correspondence with} $\mathcal{Q}.$ 
\end{definition}

This relation is fundamental for our later constructions - roughly, two tight set partitions arising from different tight set decompositions encode `the same' \BoB\ precisely when they are in correspondence. 

To begin to make this more precise, we first note some basic properties of correspondences of tight set partitions.

\begin{lemma}\label{lem:passatedge}
    Let $\rho \colon \mathcal{P} \to \mathcal{Q}$ be a correspondence of tight set partitions of $G$ and let $P$ and $P'$ be distinct elements of $\mathcal{P}.$
    Then $P \cap \rho(P')$ is passable from $P$ to $P'.$
    If it is nonempty then $P$ and $P'$ are neighbours in $\collapse{\mathcal{P}}.$ 
\end{lemma}
\begin{proof}
    To show passability, we have to show that the following four sets are tight:
    \[
        P \setminus (P \cap \rho(P')), \,\, P \cup (P \cap \rho(P')), \,\, P' \setminus (P \cap \rho(P')) \text{ and } P' \cup (P \cap \rho(P'))
    \]
    The second and third of these are just $P$ and $P'.$
    The first is $P \setminus \rho(P')$, which is tight, since $\rho(P') \cap P$ is even by definition of $\rho$, thus $P \cap (\V{G} \setminus \rho(P'))$ is odd and then by \cref{lem:basictight} $P \cap (\V{G} \setminus \rho (P')) = P \setminus \rho(P')$ is tight.
    To show that the fourth is tight, we first note that it is equal to 
    \[
        (P' \cup \rho(P')) \setminus \bigcup_{P'' \in \mathcal{P} \setminus \Set{P,P'}} P''\,,
    \]
    which is tight by \cref{lem:deletemany} since $P' \cup \rho(P')$ is tight by \cref{lem:basictight}, since $P'\cap \rho (P')$ is odd by the definition of $\rho.$

    Now suppose that $P \cap \rho(P')$ is nonempty.
    Applying \cref{lem:basictight} to $P' \cup (P \cap \rho(P'))$ and the complement of $P$ we see that there is an edge from $P \cap \rho(P')$ to $P'$, which witnesses that $P$ and $P'$ are neighbours in $\collapse{\mathcal{P}}.$
\end{proof}

\begin{lemma}\label{lem:corrisiso}
  Let $\rho \colon \mathcal{P} \to \mathcal{Q}$ be a correspondence of tight set partitions of $G.$
  Then $\rho$ is also a graph isomorphism from $\collapse{\mathcal{P}}$ to $\collapse{\mathcal{Q}}.$  
\end{lemma}
\begin{proof}
    First we show that $\rho^{-1}$ is a graph homomorphism.
    Let $e$ be an edge joining $\rho(P)$ to $\rho(P')$ in $\collapse{\mathcal{Q}}.$
    Let $v$ be the endpoint of $e$ in $\rho(P)$ and $w$ the endpoint in $\rho(P').$ 

    Suppose for a contradiction that $v$ is in an element $P''$ of $\mathcal{P}$ other than $P$ and $P'.$
    Applying \cref{lem:basicnoedge} to $P''$, $\rho(P)$ and $\rho(P')$, we see that $w$ cannot also be in $P''.$ But then we get the desired contradiction by applying \cref{lem:basictight} to $P''$ and the complement of $\rho(P).$
    
    Thus $v$ must lie in $P \cup P'$, and a symmetric argument shows that $w$ is also in $P \cup P'.$
    If $w \in P$ then $P \cap \rho(P')$ is nonempty and so by \cref{lem:passatedge} $P$ and $P'$ are neighbours in $\collapse{\mathcal{P}}.$
    The same result follows if $v$ is in $P'$, since then $P' \cap \rho(P)$ is nonempty.
    If neither of these happens then $v$ is in $P$ and $w$ is in $P'$, so $e$ itself witnesses that $P$ and $P'$ are neighbours in $\collapse{\mathcal{P}}.$
    So in any case they are neighbours, completing the proof that $\rho^{-1}$ is a graph homomorphism.
    
    Since $\rho^{-1}$ is also a correspondence from $\mathcal{Q}$ to $\mathcal{P}$, the same argument applied to $\rho^{-1}$ shows that $\rho$ is also a graph homomorphism, and thus a graph isomorphism.
\end{proof}

\section{Torsoids}
\label{sec:torsoids}

The relation of being in correspondence is not an equivalence relation, as illustrated in \cref{fig:hexagons}.
However, our next aim is to show that its restriction to some important classes of tight set partitions is an equivalence relation, and to introduce some canonical objects, called torsoids, displaying the equivalence classes.

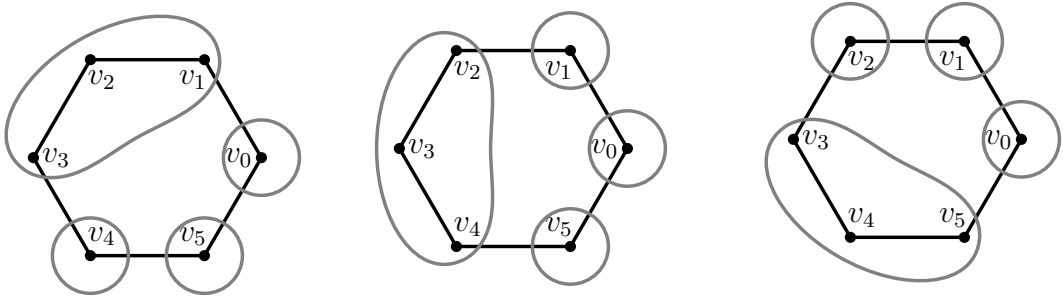
\begin{figure}
    \centering
    \begin{tikzpicture}
        \def\radius{1.5}
        \def\dist{0.5}
        \def\largedist{0.2}
        \def\smalldist{0.3}
        \node (left) at (-5,0) {
            \begin{tikzpicture}
                \node (center) at (0,0) {};
                \foreach \i in {0,...,5}
                {
                    \node[vertex,scale=0.7] (v-\i) at ($(center)+({360/6*\i}:\radius)$) {};
                    \node (v-\i-l) at ($(v-\i)+({180+360/6*\i}:0.3)$) {$v_{\i}$};
                }
                \foreach \i in {0,...,5}
                {
                    \pgfmathtruncatemacro{\j}{mod(\i+1,6)}%
                    \draw[edge] (v-\i.center) to (v-\j.center);
                }
                \foreach \i in {0,4,5}
                {
                    \draw ($(v-\i)+({90}:\dist)$) edge[edge,myGrey,curve through={($(v-\i)+({180}:\dist)$) ($(v-\i)+({270}:\dist)$)},closed] ($(v-\i)+({0}:\dist)$);
                }
                \begin{scope}
                \pgfmathtruncatemacro{\i}{1}
                \pgfmathtruncatemacro{\ii}{mod(\i-1,6)}
                \pgfmathtruncatemacro{\iii}{mod(\i-2,6)}
                \pgfmathtruncatemacro{\j}{2}
                \pgfmathtruncatemacro{\k}{3}
                \pgfmathtruncatemacro{\kk}{mod(\k+1,6)}
                \draw ($(v-\j)+({360/6*\j}:\smalldist)$) 
                    edge[edge,myGrey,curve through={
                    ($(v-\i)+({360/6*\ii}:\largedist)$)
                    ($(v-\i)+({360/6*\iii}:\smalldist)$) 
                    ($(v-\j)+({180+360/6*\j}:{\radius-\smalldist})$) 
                    ($(v-\k)+({360/6*\iii}:\smalldist)$) },closed] 
                    ($(v-\k)+({360/6*\kk}:\largedist)$);
                \end{scope}
            \end{tikzpicture}
        };
        \node (middle) at (0,0) {
            \begin{tikzpicture}
                \node (center) at (0,0) {};
                \foreach \i in {0,...,5}
                {
                    \node[vertex,scale=0.7] (v-\i) at ($(center)+({360/6*\i}:\radius)$) {};
                    \node (v-\i-l) at ($(v-\i)+({180+360/6*\i}:0.3)$) {$v_{\i}$};
                }
                \foreach \i in {0,...,5}
                {
                    \pgfmathtruncatemacro{\j}{mod(\i+1,6)}%
                    \draw[edge] (v-\i.center) to (v-\j.center);
                }
                \foreach \i in {0,1,5}
                {
                    \draw ($(v-\i)+({90}:\dist)$) edge[edge,myGrey,curve through={($(v-\i)+({180}:\dist)$) ($(v-\i)+({270}:\dist)$)},closed] ($(v-\i)+({0}:\dist)$);
                }
                \begin{scope}
                \pgfmathtruncatemacro{\i}{2}
                \pgfmathtruncatemacro{\ii}{mod(\i-1,6)}
                \pgfmathtruncatemacro{\iii}{mod(\i-2,6)}
                \pgfmathtruncatemacro{\j}{3}
                \pgfmathtruncatemacro{\k}{4}
                \pgfmathtruncatemacro{\kk}{mod(\k+1,6)}
                \draw ($(v-\j)+({360/6*\j}:\smalldist)$) 
                    edge[edge,myGrey,curve through={
                    ($(v-\i)+({360/6*\ii}:\largedist)$)
                    ($(v-\i)+({360/6*\iii}:\smalldist)$) 
                    ($(v-\j)+({180+360/6*\j}:{\radius-\smalldist})$) 
                    ($(v-\k)+({360/6*\iii}:\smalldist)$) },closed] 
                    ($(v-\k)+({360/6*\kk}:\largedist)$);
                \end{scope}
            \end{tikzpicture}
        };
        \node (right) at (5,0) {
            \begin{tikzpicture}
                \node (center) at (0,0) {};
                \foreach \i in {0,...,5}
                {
                    \node[vertex,scale=0.7] (v-\i) at ($(center)+({360/6*\i}:\radius)$) {};
                    \node (v-\i-l) at ($(v-\i)+({180+360/6*\i}:0.3)$) {$v_{\i}$};
                }
                \foreach \i in {0,...,5}
                {
                    \pgfmathtruncatemacro{\j}{mod(\i+1,6)}%
                    \draw[edge] (v-\i.center) to (v-\j.center);
                }
                \foreach \i in {0,1,2}
                {
                    \draw ($(v-\i)+({90}:\dist)$) edge[edge,myGrey,curve through={($(v-\i)+({180}:\dist)$) ($(v-\i)+({270}:\dist)$)},closed] ($(v-\i)+({0}:\dist)$);
                }
                \begin{scope}
                \pgfmathtruncatemacro{\i}{3}
                \pgfmathtruncatemacro{\ii}{mod(\i-1,6)}
                \pgfmathtruncatemacro{\iii}{mod(\i-2,6)}
                \pgfmathtruncatemacro{\j}{4}
                \pgfmathtruncatemacro{\k}{5}
                \pgfmathtruncatemacro{\kk}{mod(\k+1,6)}
                \draw ($(v-\j)+({360/6*\j}:\smalldist)$) 
                    edge[edge,myGrey,curve through={
                    ($(v-\i)+({360/6*\ii}:\largedist)$)
                    ($(v-\i)+({360/6*\iii}:\smalldist)$) 
                    ($(v-\j)+({180+360/6*\j}:{\radius-\smalldist})$) 
                    ($(v-\k)+({360/6*\iii}:\smalldist)$) },closed] 
                    ($(v-\k)+({360/6*\kk}:\largedist)$);
                \end{scope}
            \end{tikzpicture}
        };
    \end{tikzpicture}
    \caption{The tight set partitions on the left and the right are both in correspondence with the one in the middle.
    But they are not in correspondence with each other.
    Thus correspondence is not transitive.}
    \label{fig:hexagons}
\end{figure}

The example in \cref{fig:hexagons} shows that being in correspondence is not even an equivalence relation when restricted to tight set partitions whose collapses are \BoBs, since the collapses of the tight set partitions in the figure are all cycles of length 4.
However, the cycle of length 4 is the only \BoB which is problematic in this sense; we will show that being in correspondence is an equivalence relation on all other tight set partitions that collapse to \BoBs.

The problem with the tight set partitions in \cref{fig:hexagons} is that they can be refined to larger cycles.
More precisely, we say that a tight set partition $\mathcal{P}$ is \emph{cyclic} if $\collapse{\mathcal{P}}$ is a cycle, and \emph{maximal cyclic} if it is cyclic but cannot be refined to a finer cyclic tight set partition.
In the first part of this \namecref{sec:torsoids} we show that being in correspondence is an equivalence relation on maximal cyclic tight set partitions.
We call a tight set partition that either is maximal cyclic or whose collapse is a \BoB other than $C_4$ \emph{\TorsoidInducing}.
The \TorsoidInducing tight set partitions, as their name suggests, play a central role in the theory that is built throughout this paper.

For the rest of the paper let $G$ be a fixed matching covered graph.
As a preliminary step, we show that this class of tight set partitions is closed under correspondence.
First we establish a special case.

\begin{lemma}\label{lem:onestepmax}
    Let $\mathcal{P}$ be a maximal cyclic tight set partition of $G$ and let $P$ and $Q$ be consecutive vertices on the cycle $\collapse{\mathcal{P}}.$
    Let $S \subseteq P$ be a passable set between $P$ and $Q.$
    Then $\mathcal{P}' \coloneqq (\mathcal{P} \setminus \Set{P,Q}) \cup \Set{P \setminus S, Q \cup S}$ is also maximal cyclic.
\end{lemma}
\begin{proof}
    Suppose not for a contradiction.
    It is certainly cyclic by \cref{lem:corrisiso}, so it cannot be maximal.
    Choose a proper cyclic refinement $\mathcal{Q}'$ with as few vertices as possible.
    We will construct a tight set partition $\mathcal{Q}$ refining $\mathcal{P}$ which is in correspondence with $\mathcal{Q}'$ and therefore also cyclic by \cref{lem:corrisiso}.
    This contradicts the fact that $\mathcal{P}$ is maximal cyclic, giving the desired contradiction.

    Since all elements of $\mathcal{Q}$ and $\mathcal{Q}'$ are odd sets, the minimality of $\mathcal{Q}'$ implies $|\mathcal{Q}'| = |\mathcal{Q}| + 2.$
    Thus, $\mathcal{Q}'$ must be of the form $(\mathcal{P}' \setminus \Set{X}) \cup \Set{X_1,X_2,X_3}$ for some $X \in \mathcal{P}'.$
    If $X$ is not equal to either of $P \setminus S$ or $Q \cup S$ then we can simply set 
    \[
        \mathcal{Q} \coloneqq (\mathcal{Q}' \setminus \Set{P \setminus S, Q \cup S}) \cup \Set{P,Q}\,.
    \]
    
    Thus $X$ is one of $P \setminus S$ or $Q \cup S.$
    Suppose first of all that it is $P \setminus S.$
    Let $R$ be the neighbour of $Q$ other than $P$ in $\collapse{\mathcal{P}}$, and suppose without loss of generality that $X_3$ is a neighbour of $Q \cup S$ in $\mathcal{Q}'.$
    Then $X_3 \cup S$ is tight by \cref{lem:basictight} applied to $X_3 \cup (Q \cup S) \cup R$ and $P$, so we can set 
    \[
        \mathcal{Q} \coloneqq (\mathcal{Q}' \setminus \Set{X_3, Q \cup S}) \cup \Set{X_3 \cup S, Q}\,.
    \]
    
    Now assume instead that $X$ is $Q \cup S.$
    Once again, let $R$ be the neighbour of $Q$ other than $P$ in $\collapse{\mathcal{P}}.$
    Suppose without loss of generality that $X_1$ is a neighbour of $P \setminus S$ and $X_3$ is a neighbour of $R$ in $\mathcal{Q}'.$
    Let $e$ be an edge with one endpoint $v$ in $X_3$ and the other endpoint in $R.$
    By \cref{lem:basictight} applied to $V(G) \setminus P$ and $Q \cup S$, $v$ cannot lie in $X_3 \cap S$ and so must lie in $X_3 \cap Q.$ If $X_3 \setminus Q$ is odd then applying \cref{lem:basictight} to $X_3$ and the complement of $Q$ contradicts the existence of $e.$ So $X_3 \setminus Q$ must be even, meaning that $X_3 \cap Q$ is odd. 
    
    Since $(X_1 \cap Q) \cup (X_2 \cap Q) \cup (X_3 \cap Q) = Q$ is odd, it follows that $(X_1 \cap Q) \cup (X_2 \cap Q)$ is even and so $X_1 \cap Q$ and $X_2 \cap Q$ have the same parity.
    There are now two cases, according to whether both are odd or both are even.
    
    If both are odd we can note that all sets $X_i \cap Q$ are tight by \cref{lem:basictight} and so we can set 
    \begin{equation*}
        \mathcal{Q} \coloneqq (\mathcal{Q'} \setminus \Set{P \setminus S, X_1, X_2, X_3}) \cup \Set{P, X_1 \cap Q, X_2 \cap Q, X_3 \cap Q}\,.
    \end{equation*}
    
    If both are even then $X_1 \cap S$ is tight by \cref{lem:basictight} applied to $X_1$ and $P$, and $S \setminus X_1$ is tight by \cref{lem:basictight} applied to $P$ and $X_2 \cup X_3 \cup R.$
    So we can set
    \begin{equation*}
        \mathcal{Q} \coloneqq (\mathcal{Q}' \setminus \Set{X_1, X_2, X_3}) \cup \Set{X_1 \cap S, S \setminus X_1, Q}\,. \qedhere
    \end{equation*}
\end{proof}

\begin{lemma}\label{lem:max}
   Let $\mathcal{P}$ and $\mathcal{Q}$ be tight set partitions of $G$ such that there is a correspondence $\rho \colon \mathcal{P} \to \mathcal{Q}.$
   If $\mathcal{P}$ is maximal cyclic then so is $\mathcal{Q}.$ 
\end{lemma}
\begin{proof}
    Since $\mathcal{P}$ is cyclic, we can enumerate its elements as $P_1$, $P_2$, \ldots $P_n$ for some $n.$ For any $i$, let $Q_i \coloneqq \rho(P_i).$
    We will argue by induction on $\vert\Set{(i,j) \in [n]^2 \colon i \neq j \text{ but } P_i \cap Q_j \neq \emptyset}\vert.$
    The base case is that this set is empty.
    In that case for any $i \leq n$ we must have $P_i \cap Q_j = \emptyset$ for all $j \neq i$ and so $P_i \subseteq Q_i.$
    Similarly $Q_i \cap P_j = \emptyset$ for all $j \neq i$ and so $Q_i \subseteq P_i.$
    Thus in this case $P_i = Q_i$ for all $i \leq n$ and so $\mathcal{Q} = \mathcal{P}$, giving the desired result. 
    
    For the induction step, let $(i,j) \in [n]^2$ be such that $i \neq j$ but $P_i \cap Q_j \neq \emptyset.$
    By \cref{lem:passatedge} the set $S \coloneqq P_i \cap Q_j$ is passable from $P_i$ to $P_j$, and $P_i$ and $P_j$ are consecutive on $\collapse{\mathcal{P}}.$
    We set $P_i' \coloneqq P_i \setminus S$, $P_j' \coloneqq P_j \cup S$ and $P_k' \coloneqq P_k$ for all other $k.$ Let $\mathcal{P}' \coloneqq \Set{P'_i \colon i \leq n}.$ By \cref{lem:onestepmax}, $\mathcal{P}'$ is also maximal cyclic. For any $(i', j') \in [n]^2$ other than $(i, j)$ we have $P'_{i'} \cap Q_{j'} = P_{i'} \cap Q_{j'}$, whereas $P'_i \cap Q_j = \emptyset.$ So we are done by applying the induction hypothesis to $\mathcal{P}'$, $\mathcal{Q}$ and $\rho'.$
\end{proof}

Now we return to the construction of the objects displaying the correspondence equivalence classes.
The construction relies on the following key lemma.

\begin{lemma}\label{lem:disjointpass}
    Let $\mathcal{P}$ be a \TorsoidInducing tight set partition of $G.$
    Let $P$, $Q$ and $Q'$ be distinct elements of $\mathcal{P}.$
    Let $S$ be passable between $P$ and $Q$ and let $S'$ be passable between $P$ and $Q'.$
    Then $S \cap S' = \emptyset.$
\end{lemma}
\begin{proof}
    First we suppose for a contradiction that $S \cap S'$ is odd.
    We begin by establishing some notation.
    Let $P_1 \coloneqq P \setminus S'$, $P_2 : = S \cap S'$ and $P_3 \coloneqq (P \cap S') \setminus S.$
    Thus $P_1$, $P_2$ and $P_3$ partition $P.$
    All three of these are tight sets.
    $P_1$ is tight by passability of $S'.$
    $P_2$ is tight by \cref{lem:basictight} applied to $Q \cup S$ and $Q' \cup S'.$
    Finally, $P \setminus S$ is tight because $S$ is passable for $P$, and $P_3$ is tight by \cref{lem:basictight} applied to $P \setminus S$ and the complement of $P_1.$
    
    We know that $P = P_1 \cup P_2 \cup P_3$ is tight.
    We can show that also $Q \cup P_1 \cup P_2$ is tight by applying \cref{lem:basictight} to $Q \cup S$ and $P_1.$
    Finally, we know that $P_2 \cup P_3 \cup Q' = Q' \cup S'$ is tight.
    
    There are now three cases:
    \begin{description}
        \item[Case 1: $\collapse{\mathcal{P}}$ is a \BoB\ with more than four vertices.]
            Applying \cref{lem:basictight} to $Q \cup S$ and $Q' \cup S'$ shows that the set $Q \cup S \cup Q' \cup S'$ is tight.
            But $\oddIntersections{\mathcal{P}}{Q \cup S \cup Q' \cup S'} = \Set{P, Q, Q'}$ is neither a singleton nor the complement of a singleton in $\mathcal{P}$, contradicting \cref{lem:cutswell}.
        \item[Case 2: $\collapse{\mathcal{P}}$ is a \BoB\ with four vertices but is not a cycle.]
            In this case, applying \cref{lem:extendcycle} to $\collapse{(\mathcal{P} \setminus \Set{P}) \cup \Set{P_1, P_2, P_3}}$ shows that this graph is a cycle of length $6$, and hence $\collapse{\mathcal{P}}$ is a cycle of length $4$, contradicting our assumptions.
        \item[Case 3: $\mathcal{P}$ is maximal cyclic.]
            Applying \cref{lem:extendcycle} to $\collapse{(\mathcal{P} \setminus \Set{P}) \cup \Set{P_1, P_2, P_3}}$ shows that this graph is a cycle, contradicting maximality of $\mathcal{P}.$
    \end{description}
    Since we reached a contradiction in all three cases, our assumption that $S \cap S'$ is odd must have been false.
    So $S \cap S'$ is even.
    
    By \cref{lem:basictight}
    applied to $Q \cup S$ and the complement of $Q' \cup S'$, we know that $Q \cup (S \setminus S')$ is tight.
    Similarly $Q' \cup (S' \setminus S)$ is tight.
    Applying \cref{lem:deletemany} to $P$, $Q \cup (S\setminus S')$ and $Q' \cup (S' \setminus S)$ yields that $P' \coloneqq (P \setminus (S \cup S')) \cup (S \cap S')$ is tight.
    Applying \cref{lem:nothree} to $Q \cup S$, $Q' \cup S'$ and $P'$, we see that their intersection $S \cap S'$ is empty as required. 
\end{proof}

This allows us to define canonical objects which determine the correspondence equivalence classes for such tight set partitions.

\begin{definition}[torsoid]
    \label{def:torsoid}
    \label{def:correspondence}
    A \emph{torsoid} $\mathcal{T}$ in a matching covered graph $G$ is a pair $\Brace{H,\edgefkt{}}$ with:
    \renewcommand{\labelenumi}{\textbf{\theenumi}}
	\renewcommand{\theenumi}{(T\arabic{enumi})}
	\begin{enumerate}[labelindent=0pt,labelwidth=\widthof{\ref{last-item-torsoid}},itemindent=1em]
        \item \label{torsoid-def-1} $H$ is a matching covered graph on at least 4 vertices that is a \BoB\ or a cycle
        \item \label{torsoid-def-2} the elements of $V(H)$ are tight sets of $G$
        \item \label{torsoid-def-3}  $\edgefkt{}: \E{H} \to 2^{\V{G}}$
        \item \label{torsoid-def-4} $V(H) \cup \image{\edgefkt{}}$ is a near partition of $\V{G},$ where $\image{\edgefkt{}}$ is the image of $\edgefkt{}$
        \item \label{torsoid-def-5} for $vw \in \E{H}$, there is an edge from $\vertexfkt{v} \cup \edgefkt{vw}$ to $\vertexfkt{w}$ and  $\edgefkt{vw}$ is largest among all subsets of $\vertexfkt{v} \cup \edgefkt{vw} \cup \vertexfkt{w}$ that are passable for both $\vertexfkt{v}$ and $\vertexfkt{w}$
        \item \label{torsoid-def-6} for $vw \notin \E{H}$, there is no edge from $\vertexfkt{v}$ to $\vertexfkt{w}$ in $G$
        \item \label{torsoid-def-7} if $H$ is a cycle and $v$ is a vertex of $H$ with neighbours $u$ and $w$ then there is no partition of $\edgefkt{uv} \cup \vertexfkt{v} \cup \edgefkt{vw}$ into tight sets $P_1,P_2,P_3$ such that both $\vertexfkt{u} \cup P_1 \cup P_2$ and $P_2 \cup P_3 \cup \vertexfkt{w}$ are tight.
            \label{last-item-torsoid}
    \end{enumerate}
    
    We call $\mathcal{T}$ \emph{cyclic} if $H$ is a cycle and \emph{noncyclic} otherwise.
    
    A \emph{correspondence} between $\mathcal{T}$ and a tight set partition $\mathcal{P}$ is a bijection $\sigma \colon V(H) \to \mathcal{P}$ such that for every $v \in V(H)$ we have 
    \[
        v \subseteq \sigma(v) \subseteq v \cup \bigcup_{w \in N_H(v)} \edgefkt{vw}\,.
    \]
    We call $\sigma$ \emph{strong} if for every vertex $v$ of $H$ the set $\sigma(v)$ is a union of $v$ and some of the sets $\edgefkt{vw}$ with $w$ a neighbour of $v$ in $H.$
    
    We say $\mathcal{P}$ is in (strong) correspondence with $\mathcal{T}$ if there is such a (strong) correspondence $\sigma.$
\end{definition}

There is another slightly different perspective on the set of tight set partitions in strong correspondence with $\mathcal{T}.$
For any graph $H$, we say a function $\kappa \colon E(H) \to V(H)$ is a \emph{choice function} for $H$ if $\kappa(e)$ is an endvertex of $e$ for all edges $e$ of $H.$
For any pair $\mathcal{T} = (H,\edgefkt{})$ of a graph and a function defined on its edge set and any choice function $\kappa$ for $H$, we take $\mathcal{P}(\mathcal{T},\kappa)$ to be the set of all sets of the form $v \cup \bigcup_{\kappa(e) = v} \edgefkt{e}.$

\begin{lemma}
    \label{choicetostrong}
    For any torsoid $\mathcal{T} = (H, \edgefkt{})$ in $G$ and any choice function $\kappa$ for $H$, the set $\mathcal{P}(\mathcal{T},\kappa)$ is a tight set partition which is in strong correspondence with $\mathcal{T}.$
    Furthermore, all tight set partitions in strong correspondence with $\mathcal{T}$ arise in this way.
\end{lemma}
\begin{proof}
    All elements of $\mathcal{P}(\mathcal{T}, \kappa)$ are tight by \cref{torsoid-def-5} and \cref{lem:infunion}, and they partition $\V{G}$ by \cref{torsoid-def-4}.
    The function $\sigma$ sending each element of $\mathcal{P}(\mathcal{T},\kappa)$ to the unique vertex of $H$ which it includes is a strong correspondence. 
    
    For the last part, given a strong correspondence $\sigma$ between $\mathcal{T}$ and a tight set partition $\mathcal{P}$, let $\kappa$ be the function sending each edge $e$ of $H$ to the unique vertex $v$ with $\edgefkt{e} \subseteq \sigma(v).$
    Then $\mathcal{P} = \mathcal{P}(\mathcal{T},\kappa).$
\end{proof}

\begin{proposition} \label{prop:edges_of_passable_set}
    Let $\mathcal{T}=\Torsoid{H}$ be a torsoid in $G$ and $uv \in E(H).$
    Any edge in $G$ with an endvertex in $\edgefkt{uv}$ has both endvertices in $\vertexfkt{u} \cup \vertexfkt{v} \cup \edgefkt{uv}.$ 
\end{proposition}
\begin{proof}
    This follows by applying \cref{lem:basictight} to $\vertexfkt{u} \cup \edgefkt{uv}$ and the complement of $\edgefkt{uv} \cup \vertexfkt{v}.$
\end{proof}

\begin{lemma}\label{lem:strongcorrisiso}
    If $\sigma$ is a strong correspondence between a torsoid $\mathcal{T} = \Torsoid{H}$ in $G$ and a tight set partition $\mathcal{P}$ of $G$ then $\sigma$ is a graph isomorphism from $H$ to $\collapse{\mathcal{P}}.$
\end{lemma}
\begin{proof}
    This is immediate from \cref{torsoid-def-5}, \cref{torsoid-def-6} and \cref{prop:edges_of_passable_set}.
\end{proof}

We can now show that any tight set partition of the kinds discussed above induces a torsoid.

\begin{definition}\label{def:inducedtorsoid}
    Let $\mathcal{P}$ be a \TorsoidInducing tight set partition of $G.$
    For any edge $PQ$ of $\collapse{\mathcal{P}}$ we define $\delta_{\mathcal{P}}(PQ)$ to be the largest passable set between $P$ and $Q$ (this exists by \cref{lem:maxpass}).
    For $P \in \mathcal{P}$, we define $\tau_{\mathcal{P}}(P)$ to be $P \setminus \bigcup_{Q \in N_{\collapse{\mathcal{P}}}(P)}\delta(PQ).$ 
    
    Let $H_{\mathcal{P}}$ be the unique graph on the image of $\tau_{\mathcal{P}}$ making $\tau_{\mathcal{P}}$ a graph isomorphism from $\collapse{\mathcal{P}}$ to $H_{\mathcal{P}}.$
    Let $\sigma_{\mathcal{P}}$ be the inverse of $\tau_{\mathcal{P}}.$
    Let $\varepsilon_{\mathcal{P}} \colon E(H) \to 2^{\V{G}} ; vw \mapsto \delta_{\mathcal{P}}(\sigma_{\mathcal{P}}(v)\sigma_{\mathcal{P}}(w)).$
    Then we call $\mathcal{T}_{\mathcal{P}} \coloneqq \Brace{H_{\mathcal{P}},\varepsilon_{\mathcal{P}}}$ the \emph{induced torsoid} of $\mathcal{P}.$
\end{definition}

\begin{theorem}\label{thm:inducedtorsoid}
    In the context of \cref{def:inducedtorsoid}, $\mathcal{T}_{\mathcal{P}}$ is a torsoid and $\sigma_{\mathcal{P}}$ is a correspondence.
\end{theorem}
\begin{proof}
   Conditions \cref{torsoid-def-1,torsoid-def-3,torsoid-def-6} are clear from the construction.
   \Cref{torsoid-def-2} follows from \cref{lem:disjointpass} and \cref{lem:deletemany}.
   \Cref{torsoid-def-4} follows from \cref{lem:disjointpass}.
   
   For \cref{torsoid-def-5}, first we note that passability of $\varepsilon_{\mathcal{P}}(vw)$ follows from \cref{lem:basictight} together with \cref{lem:modifypass}.
   Similarly any set $S \subseteq v \cup \varepsilon_{\mathcal{P}}(vw) \cup w$ which is passable for both $v$ and $w$ is also passable between $\sigma_{\mathcal{P}}(v)$ and $\sigma_{\mathcal{P}}(w)$ by \cref{lem:modifypass}, hence $\varepsilon_{\mathcal{P}}(vw)$ is maximal amongst such sets.
   To see that there is an edge from $v \cup \varepsilon_{\mathcal{P}}(vw)$ to $w$, let $e$ be any edge from $\sigma_{\mathcal{P}}(v)$ to $\sigma_{\mathcal{P}}(w)$ in $G.$
   For any $u \in N_{H_{\mathcal{P}}}(v) \setminus \Set{w}$, the endpoint of $e$ in $\sigma_{\mathcal{P}}(v)$ cannot be in $\varepsilon_{\mathcal{P}}(uv)$ by \cref{lem:basictight} applied to $u \cup \varepsilon_{\mathcal{P}}(uv)$ and the complement of $\varepsilon_{\mathcal{P}}(uv) \cup v.$
   So it must be in $v \cup \varepsilon_{\mathcal{P}}(vw).$
   Similarly the other endpoint of $e$ must be in $\varepsilon_{\mathcal{P}}(vw) \cup w.$
   If either of these endpoints is in $\varepsilon_{\mathcal{P}}(vw)$ then we get the desired edge by \cref{lem:basictight} applied to $v \cup \varepsilon_{\mathcal{P}}(vw)$ and the complement of $\varepsilon_{\mathcal{P}}(vw) \cup w.$
   Otherwise we can take $e$ itself as the desired edge.
   
   For \cref{torsoid-def-7}, suppose $H_{\mathcal{P}}$ is a cycle and let $\kappa$ be a choice function for $H_{\mathcal{P}}$ such that $\kappa^{-1}(u)$ and $\kappa^{-1}(w)$ are empty.
   Then $\mathcal{P}(\mathcal{T},\kappa)$ is a tight set partition by \cref{torsoid-def-5} and \cref{lem:basictight} applied to the sets of the form $\kappa(e) \cup \edgefkt{e}$, and is in correspondence with $\mathcal{P}$ by construction.
   By \cref{torsoid-def-5,torsoid-def-6}, $\collapse{\mathcal{P}(\mathcal{T},\kappa)}$ is a cycle, and hence so is $\mathcal{P}$ by \cref{lem:corrisiso}.
   So $\mathcal{P}$ is cyclic, and by assumption it must be maximal cyclic.
   Thus $\mathcal{P}(\mathcal{T},\kappa)$ is also maximal cyclic by \cref{lem:max}.
   But sets $P_1$, $P_2$ and $P_3$ as in \cref{torsoid-def-7} would contradict maximality due to \cref{lem:extendcycle}.
   So they cannot exist.
   
   Finally, it follows directly from the construction that $\sigma_{\mathcal{P}}$ is a correspondence from $\mathcal{T}_\mathcal{P}$ to $\mathcal{P}.$ 
\end{proof}

\begin{theorem}\label{thm:corriseq}
    Let $\mathcal{P}$ and $\mathcal{Q}$ be \TorsoidInducing tight set partitions of $G.$
    Then $\mathcal{P}$ and $\mathcal{Q}$ are in correspondence if and only if $\mathcal{T}_{\mathcal{P}} = \mathcal{T}_{\mathcal{Q}}.$
    In particular, being in correspondence is an equivalence relation for such tight set partitions.
\end{theorem}
\begin{proof}
    Suppose first that there is a correspondence $\rho \colon \mathcal{P} \to \mathcal{Q}.$
    Then, by \cref{lem:passatedge}, any set of the form $P \cap \rho(Q)$ with $P \neq Q$ is passable between $P$ and $Q.$
    Using \cref{lem:modifypass} twice, we may therefore show first that $\delta_{\mathcal{P}}(PQ)$ is passable for $P \cap \rho(P) = P \setminus \bigcup_{R \in N_{\collapse{\mathcal{P}}}(P) \setminus \Set{Q}} (P \cap \rho(R))$ and then that it is passable for $\rho(P).$
    A similar argument shows that it is passable for $\rho(Q)$, so we have $\delta_{\mathcal{P}}(PQ) \subseteq \delta_{\mathcal{Q}}(\rho(P)\rho(Q)).$
    A similar argument applied to $\rho^{-1}$ implies the reverse inclusion, so we have $\delta_{\mathcal{P}}(PQ) = \delta_{\mathcal{Q}}(\rho(P)\rho(Q)).$ 
    
    Now let $x$ be any element of $\tau_{\mathcal{P}}(P)$ for any $P \in \mathcal{P}.$
    Let $Q$ be the element of $\mathcal{P}$ with $x \in \rho(Q).$
    Then $Q$ cannot be different from $P$, since then it would have to be a neighbour of $P$ in $\collapse{\mathcal{P}}$ by \cref{lem:passatedge}, so we would have $x \in P \cap \rho(Q) \subseteq \delta_{\mathcal{P}}(PQ)$, contradicting the definition of $\tau_{\mathcal{P}}(P).$
    Thus $x \in \rho(P).$
    By construction, for any neighbour $\rho(Q)$ of $\rho(P)$ in $\collapse{\mathcal{Q}}$ we have $x \not \in \delta_{\mathcal{P}}(PQ) = \delta_{\mathcal{Q}}(\rho(P)\rho(Q)).$
    Thus $x \in \tau_{\mathcal{Q}}(\rho(P)).$
    This argument shows that $\tau_{\mathcal{P}}(P) \subseteq \tau_{\mathcal{Q}}(\rho(P)).$
    A similar argument applied to $\rho^{-1}$ implies the reverse inclusion, so we have $\tau_{\mathcal{P}}(P) = \tau_{\mathcal{Q}}(\rho(P))$ for any $P.$ 
    
    This immediately implies that $H_{\mathcal{P}}$ and $H_{\mathcal{Q}}$ are equal.
    Finally, for any edge $vw$ of $H_{\mathcal{P}}$ with $v = \tau_{\mathcal{P}}(P)$ and $w = \tau_{\mathcal{P}}(Q)$, we have 
    \[
        \varepsilon_{\mathcal{P}}(vw) = \delta_{\mathcal{P}}(PQ) = \delta_{\mathcal{Q}}(\rho(P)\rho(Q)) = \varepsilon_{\mathcal{Q}}(vw)\,.
    \]
    This completes the proof that $\mathcal{T}_{\mathcal{P}} = \mathcal{T}_{\mathcal{Q}}.$
    
    For the reverse direction, suppose that $\mathcal{T}_{\mathcal{P}} = \mathcal{T}_{\mathcal{Q}}$, and let this torsoid be given by $(H,\edgefkt{}).$
    Let $\rho \coloneqq \sigma_{\mathcal{Q}} \cdot \tau_{\mathcal{P}}.$
    Let $PQ$ be any edge of $\collapse{\mathcal{P}}.$
    Set $v \coloneqq \tau_{\mathcal{P}}(P)$ and $w \coloneqq \tau_{\mathcal{P}}(Q).$
    Then by \cref{lem:deletemany} applied to $P$ and the sets $\delta_{\mathcal{P}}(P,Q')$ for all other neighbours $Q'$ of $P$, the set $v \cup (P \cap \edgefkt{vw})$ is tight.
    Similarly the set $w \cup (Q \cap \edgefkt{vw})$ is tight.
    Thus $P \cap \edgefkt{vw}$ is passable between $v$ and $\edgefkt{vw} \cup w.$
    A similar argument shows that $\rho(P) \cap \edgefkt{vw}$ is also passable between these sets.
    So by \cref{lem:notodd} their intersection $P \cap \rho(P) \cap \edgefkt{vw}$ is even.
    Since $P \cap \edgefkt{vw}$ is even, it follows that $(P \cap \edgefkt{vw}) \setminus \rho(P) = P \cap \rho(Q) \cap \edgefkt{vw}$ is also even.
    Since by construction any element of $P \cap \rho(Q)$ must lie in $\edgefkt{vw}$, this implies that $P \cap \rho(Q)$ is even. 
    
    Taking stock of the argument so far, if $Q \neq P$ is a neighbour of $P$ then $P \cap \rho(Q)$ is even.
    But if $Q$ is not a neighbour of $P$ then by construction $P \cap \rho(Q)$ is empty and so also even.
    Thus the set $\mathcal{Q}_P$ of all elements of $\mathcal{Q}$ whose intersection with $P$ is odd is either empty or equal to $\Set{\rho(P)}.$
    By \cref{lem:uncrosscut} its union is tight, so it cannot be empty.
    Thus it is $\Set{\rho(P)}.$
    That is, $\rho(P)$ is the unique element of $\mathcal{Q}$ whose intersection with $P$ is odd.
    Since this is true for all $P \in \mathcal{P}$, $\rho $ is a correspondence from $\mathcal{P}$ to $\mathcal{Q}.$
\end{proof}

\section{Relation of tight sets to torsoids}
\label{sec:relation_tigh_sets_torsoids}

In this \namecref{sec:relation_tigh_sets_torsoids} we investigate how torsoids interact with tight cuts and tight sets.
We introduce three classes of tight cuts with respect to a given torsoid and prove that these three classes partition the set of tight cuts.
This partition induces a partition of the tight sets.

\begin{definition}
    For a torsoid $\mathcal{T} = \Torsoid{H}$ in $G$ and a tight set $X \subseteq \V{G}$ we define \emph{$\oddIntersections{H}{X}$} to be the set of vertices of $H$ that have odd intersection with $X.$
    For a tight cut $C = \boundary{X}$ we define $\theta_{\mathcal{T}}(C)\coloneqq \min \Set{|\oddIntersections{H}{X}|, |\oddIntersections{H}{\V{G} \setminus X}|}.$
\end{definition}
The constant $\theta_{\mathcal{T}}(C)$ determines to which of the three classes a tight cut $C$ belongs: $C$ resides at an edge ($\theta_{\mathcal{T}}(C) = 0$), resides at a vertex ($\theta_{\mathcal{T}}(C) = 1$) or resides at an interval ($\theta_{\mathcal{T}}(C) > 1$).  Note that as $\oddIntersections{H}{\V{G} \setminus X} = V(H) \setminus \oddIntersections{H}{X}$ holds, the constant $\theta_{\mathcal{T}}(C)$ is bounded by $\frac{|V(H)|}{2}.$

\begin{lemma}\label{lem:tight_sets_contain_torsoid_vertices_whole_or_are_disjoint}
    Let $\mathcal{T} = \Torsoid{H}$ be a torsoid in $G$, $a \in \V{H}$ and $X$ a tight set in $G.$
    If $X \cap a$ is even then $X \setminus a$ is tight and if $X \cap a$ is odd then $X \cup a$ is tight.
\end{lemma}
\begin{proof}
    First suppose that $X \cap a$ is even.
    Then $\V{G} \setminus X$ and $a$ have odd intersection thus $\Brace{\V{G}\setminus X} \cup a $ is tight by \cref{lem:basictight} and so is $\V{G} \setminus (\Brace{\V{G} \setminus X} \cup a) = X \setminus a.$
    Now suppose that $X \cap a$ is odd.
    Then $X \cup a$ is tight by \cref{lem:basictight}.
\end{proof}

\begin{lemma}\label{lem:tight_sets_contain_torsoid_edges_they_meet_evenly_whole_or_are_disjoint}
    Let $\mathcal{T} = \Torsoid{H}$ be a torsoid in $G$, $ab \in \E{H}$ and $X$ a tight set in $G.$
    If $X \cap \Brace{a \cup \edgefkt{ab}}$ is odd, $X \cup a \cup \edgefkt{ab}$ is tight. If $X \cap \Brace{a \cup \edgefkt{ab}}$ is even, $X \setminus \Brace{a \cup \edgefkt{ab}}$ is tight.
\end{lemma}
\begin{proof}
    If $X \cap \Brace{a \cup \edgefkt{ab}}$ is odd, the tight sets $X$ and $a \cup \edgefkt{ab}$ intersect oddly. Thus $X \cup a \cup \edgefkt{ab}$ is tight by \cref{lem:basictight}. If $X \cap \Brace{a \cup \edgefkt{ab}}$ is even, apply the prior result to the complement of $X.$
\end{proof}

\begin{definition}[\pcut{}]
    Let $\mathcal{T}=\Torsoid{H}$ be a torsoid in $G$ and let $e$ be an edge of $H.$
    Then we say that a tight cut $C$ \emph{resides at an edge $e$ in $\mathcal{T}$} if there is a tight set $X \subseteq \edgefkt{e}$ such that $C = \boundary{X}$ and we call $X$ a \emph{\pcut{$\mathcal{T}$} at $e$}.
\end{definition}
See \cref{fig:vertex_resident} for an example of a \BoB with a \pcut{}.

\begin{lemma} \label{lem:one_edge_odd}
    Let $\mathcal{T} = \Torsoid{H}$ be a torsoid in $G.$
    Let $X \subseteq \V{G}$ be a tight set and $uv \in E(H)$ such that $X \cap \edgefkt{uv}$ is odd.
    Then $\boundary{X}$ resides at the edge $uv$ in $\mathcal{T}.$
\end{lemma}
\begin{proof}
    Let $\kappa_1$ be any choice function such that $\kappa_1(uv) = u$, and $\kappa_1(e) \notin \Set{u, v}$ for any $uv \neq e \in E(H).$ Let $\kappa_2$ be the choice function which agrees with $\kappa_1$ except at $uv$, where $\kappa_2(uv) = v.$
    Consider $\mathcal{P}^i \coloneqq \mathcal{P}(\mathcal{T}, \kappa_i)$ for $i \in [2].$
    We set $u_1\coloneqq u \cup \edgefkt{uv}, v_1\coloneqq v$ and $u_2\coloneqq u, v_2\coloneqq v \cup \edgefkt{uv}.$
    Note that $u_i, v_i \in \mathcal{P}^i$ for $i \in [2].$
    Furthermore, $\mathcal{P}^1 \setminus \Set{u_1, v_1} = \mathcal{P}^2 \setminus \Set{u_2, v_2}.$

    Let us first show that $X$ intersects either both $u, v$ evenly or both oddly.
    Without loss of generality we suppose for a contradiction that $X$ intersects $u$ oddly and $v$ evenly. There are now three cases:
    \begin{description}
        \item[$H$ is a \BoB~with more than four vertices.]
            Then $|\oddIntersections{\mathcal{P}}{X}^1|$ and $|\oddIntersections{\mathcal{P}}{X}^2|$ are elements of $\Set{1, |V(H)|-1}$ by \cref{lem:cutswell}.
            By construction, $|\oddIntersections{\mathcal{P}}{X}^1| + 2 = |\oddIntersections{\mathcal{P}}{X}^2|$ holds and this gives a contradiction since $|V(H)| \geq 6.$
        \item[$H$ is a \BoB~with four vertices but is not a cycle.]
            Note that $K^4$ is the only \BoB~on four vertices other than a cycle.
            At first, we assume that there exists $s \in \V{H}\setminus \Set{u,v}$ such that $Y_s \coloneqq s \cup \edgefkt{us} \cup u \cup \Brace{X \cap \edgefkt{uv}}$ is tight and show that this gives a contradiction. Afterwards, we show that such a vertex $s$ exists such that $Y_s$ is tight.
            \begin{figure}[ht]
                \centering
                \includegraphics[scale=1.10]{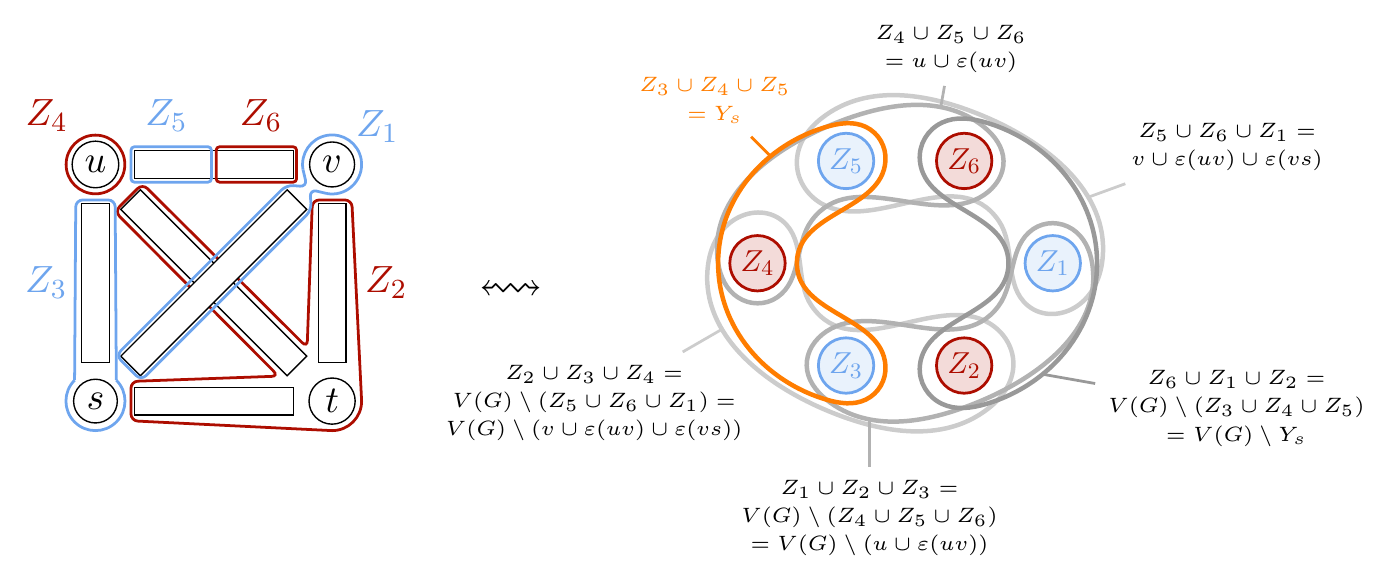}
                \caption{In the context of \cref{lem:one_edge_odd} if $H$ is a $K_4$ and $Y_s$ is a tight set as depicted, we deduce that the edge $ut$ of $H$ cannot exist for a contradiction.}
                \label{fig:passapble_cuts_in_K_4}
            \end{figure}
            Under these assumptions let $t$ be the vertex of $H$ distinct from $u, v, s$ and consider the sets $Z_1\coloneqq v \cup \edgefkt{vs}$ and $Z_2\coloneqq t \cup \bigcup_{w \in \E{H} \setminus \set{t}} \edgefkt{wt}$, $Z_3\coloneqq s \cup \edgefkt{us}$, $Z_4\coloneqq u$, $Z_5\coloneqq X \cap \edgefkt{uv}$, $Z_6\coloneqq \edgefkt{uv} \setminus X$ in this cyclic order.
            The sets $Z_1, Z_2, Z_3, Z_4$ are tight by definition. The set $Z_5$ is tight by applying \cref{lem:basictight} to $X \setminus v$ and $v \cup \edgefkt{uv}.$ The set $Z_6$ is tight by applying \cref{lem:basictight} to $\V{G} \setminus \Brace{X \cup u}$ and $u \cup \edgefkt{uv}.$

            Furthermore, $Z_3 \cup Z_4 \cup Z_5 = Y_s$ is tight by assumption. Also $Z_4 \cup Z_5 \cup Z_6 = u \cup \edgefkt{uv}$ and $Z_5 \cup Z_6 \cup Z_1 = v \cup \edgefkt{uv} \cup \edgefkt{vs}$ are tight. As $Z_6 \cup Z_1 \cup Z_2$, $Z_1 \cup Z_2 \cup Z_3$ and $Z_2 \cup Z_3 \cup Z_4$ are complements of the first three sets respectively, they are also tight.
            Thus these sets satisfy the conditions of \cref{lem:charcycle} and therefore $\collapse{\Set{Z_i: i \in [6]}}$ is a cycle. This contradicts that $ut$ is an edge of $H.$

            It remains to prove that there exist $s\in \V{H} \setminus \Set{u, v}$ such that $Y_s$ is indeed tight.
            The set $X'\coloneqq X \cup \bigcup \oddIntersections{H}{X} \setminus \bigcup \Brace{V(H) \setminus \oddIntersections{H}{X}}$ is tight by repeated application of \cref{lem:tight_sets_contain_torsoid_vertices_whole_or_are_disjoint} to $X.$ By construction, $X'$ has the property that any vertex $v \in \V{H}$ is either contained in $X'$ or disjoint from $X'.$ Next we apply \cref{lem:tight_sets_contain_torsoid_edges_they_meet_evenly_whole_or_are_disjoint} repeatedly to $X'$ and any edge $ab \in E(H)$ with $a \in \Set{s, t}$ to obtain a tight set $X''$ such that $X' \cap \Brace{u \cup v \cup \edgefkt{uv}} = X'' \cap \Brace{u \cup v \cup \edgefkt{uv}}$ and for any edge $e \in \E{H} \setminus \Set{uv}$ we have $X'' \cap \edgefkt{e} \in \Set{\emptyset, \edgefkt{e}}.$ Furthermore, any vertex $v \in \V{H}$ is also either contained in $X''$ or disjoint from $X''$ by construction.
            
            Note that $X''$ is odd, $X''$ intersects precisely one edge of $\mathcal{T}$ oddly and $X''$ intersects $u$ oddly and $v$ evenly. Thus there exists exactly one vertex $s \in \V{H} \setminus \Set{u, v}$ such that $s \subset X''.$
            Therefore the tight set $X''$ has the following properties: $X'' \cap \edgefkt{uv} = X \cap \edgefkt{uv}$, $\oddIntersections{H}{X''} = \Set{u,s}$ and for any $e \in \E{H} \setminus \Set{uv}$ we have $X'' \cap \edgefkt{e} \in \Set{\emptyset, \edgefkt{e}}.$

            We apply \cref{lem:tight_sets_contain_torsoid_edges_they_meet_evenly_whole_or_are_disjoint} to $X''$, the vertex $s$ and the edge $us$ to obtain a tight set $X'''$ with $X''' = X'' \cup \edgefkt{us}.$ Similarly we apply \cref{lem:tight_sets_contain_torsoid_edges_they_meet_evenly_whole_or_are_disjoint} repeatedly to $X'''$, the vertex $v$ with the edges $vs, vt$ and the vertex $t$ with the edge $st$ to obtain a tight set that coincides with $Y_s$, by construction. This completes this case.
            
        \item[$H$ is a cycle of length four.]
            We find a contradiction to \cref{torsoid-def-7}.
            Let $w$ be the neighbour of $v$ in $H$ distinct from $u.$
            Set $P_1 \coloneqq \edgefkt{uv} \cap X$, $P_2\coloneqq \edgefkt{uv} \setminus X$ and $P_3 \coloneqq v \cup \edgefkt{vw}.$
            By construction, $P_1, P_2, P_3$ partition $\edgefkt{uv} \cup v \cup \edgefkt{vw}.$
            The set $P_3$ is a tight set as $\edgefkt{vw}$ is passable.
            We consider the tight set $X'\coloneqq \Brace{X \cup u} \setminus v.$
            Now $P_1 = X' \cap \Brace{v \cup \edgefkt{uv}}$ and $P_2 = \Brace{\V{G} \setminus X'} \cap \Brace{u \cup \edgefkt{uv}}$ are tight sets.
            Furthermore by construction, $u \cup P_1 \cup P_2$ is a tight set.
            It remains to prove that $P_2 \cup P_3 \cup w$ is a tight set.

            We suppose without loss of generality that $\edgefkt{vw} \cup w$ is a partition class of $\mathcal{P}_1$ (otherwise modify $\mathcal{P}_1$ slightly).
            By choice, $V(G) \setminus X'$ intersects $u_1$ oddly, but does not contain $u_1.$ Furthermore, $V(G) \setminus X'$ intersects $v_1$ oddly.
            Applying \cref{lem:cutscyclicpartition} to $V(G) \setminus X'$ and $\mathcal{P}^1$, we see that $\oddIntersections{\mathcal{P}^1}{V(G) \setminus X'}$ is an interval with $u_1 = u \cup \edgefkt{uv}$ as an endpoint and containing $v.$
            Since this interval has odd length, it must also contain the element $\edgefkt{vw} \cup w$ of $\mathcal{P}^1$ containing $w.$
            By construction, $\oddIntersections{\mathcal{P}^1}{V(G) \setminus X'} \cap (u \cup \edgefkt{uv} \cup v \cup \edgefkt{vw} \cup w)$ is odd, and by \cref{lem:basictight} tight.
            Note that $\oddIntersections{\mathcal{P}^1}{V(G) \setminus X'} \cap (u \cup \edgefkt{uv} \cup v \cup \edgefkt{vw} \cup w) = P_2 \cup P_3 \cup w.$
    \end{description}
    Thus we can suppose that $X$ intersects either both $u,v$ oddly or both evenly. This implies that $X$ intersects exactly one of $u_i, v_i$ oddly for $i \in [2].$
    We prove that $X$ intersects either all elements of $\mathcal{P}^i \setminus \Set{u_i, v_i}$ evenly or all these elements oddly for $i \in [2].$
    If $H$ is a \BoB~this follows directly from \cref{lem:cutswell}.
    If $H$ is a cycle, we apply \cref{lem:cutscyclicpartition}.
    Then $\oddIntersections{\mathcal{P}^1}{X}$ and $\oddIntersections{\mathcal{P}^2}{X}$ are intervals in $\collapse{\mathcal{P}^1}$ and $\collapse{\mathcal{P}^2}.$
    As $X$ intersects $u_1$ oddly if and only if $X$ intersects $v_2$ oddly and similarly for $v_1$ and $u_2$, $X$ has to intersect the elements of $\mathcal{P}^i \setminus \Set{u_i, v_i}$ either all oddly or all evenly.

    From now on we suppose without loss of generality that $X$ intersects all elements of $\mathcal{P}^i \setminus \Set{u_i, v_i}$ evenly (otherwise consider the complement of $X$).
    Furthermore, we suppose without loss of generality that $X$ intersects $u_1$ and $v_2$ oddly (otherwise exchange $\kappa_1$ and $\kappa_2$).
    Then $\oddIntersections{\mathcal{P}}{X}^1 = \Set{u_1}$ and $\oddIntersections{\mathcal{P}}{X}^2 = \Set{v_2}$ holds.
    We define the following two tight set partitions:
    \[
        \Tilde{\mathcal{P}}^1\coloneqq \Set{P \setminus X: P \in \mathcal{P}^1 \setminus \Set{u_1}} \cup \Set{u_1 \cup X},
    \]
    \[
        \Tilde{\mathcal{P}}^2\coloneqq \Set{P \setminus X: P \in \mathcal{P}^2 \setminus \Set{v_2}} \cup \Set{v_2 \cup X}.
    \]
    By \cref{lem:basictight} these are indeed tight set partitions.
    By construction, $\Tilde{\mathcal{P}}^1$, $\mathcal{P}^1$ correspond and $\Tilde{\mathcal{P}}^2$, $\mathcal{P}^2$ correspond.
    By \cref{lem:corrisiso} and \cref{lem:max}, $\collapse{\Tilde{\mathcal{P}}^1}$ and $\collapse{\Tilde{\mathcal{P}}^1}$ are \TorsoidInducing.
    Thus we can apply \cref{thm:corriseq}, which shows that, $\mathcal{T}_{\Tilde{\mathcal{P}_1}} = \mathcal{T} = \mathcal{T}_{\Tilde{\mathcal{P}_2}}$ holds.
    This implies that $ u_1 \cup X \subseteq u \cup \bigcup_{w \in N_H(u)} \edgefkt{uw}$ and $ v_2 \cup X \subseteq v \cup \bigcup_{w \in N_H(v)} \edgefkt{vw}.$
    Then
    \[
        X \subset \Brace{ u \cup \bigcup_{w \in N_H(u)} \edgefkt{uw}} \cap \Brace{v \cup \bigcup_{w \in N_H(v)} \edgefkt{vw}} = \edgefkt{uv}
    \]
    holds.
    This proves that $X$ is an \pcut{}. Thus $\boundary{X}$ resides at an edge of $\mathcal{T}.$
\end{proof}

\begin{lemma} \label{lem:pcut_char}
    Let $\mathcal{T}$ be a torsoid in $G$ and $C \subseteq \E{G}$ any tight cut. The tight cut $C$ resides at an edge in $\mathcal{T}$ if and only if $\theta_{\mathcal{T}}(C) = 0$ holds.
\end{lemma}
\begin{proof}
    By definition, any tight cut $C$ of $G$ that resides at an edge of $H$ in $\mathcal{T}$ satisfies $\theta_{\mathcal{T}}(C) = 0.$
    For the only if direction let $C$ be a tight cut in $G$ with $\theta_{\mathcal{T}}(C) = 0.$
    Let $ X \subseteq \V{G} $ such that $C = \boundary{X}.$
    By \cref{lem:tight_sets_contain_torsoid_vertices_whole_or_are_disjoint} we can assume that $X \cap v = \emptyset$ for any $v \in V(H).$
    We show that $X$ has odd intersection with $\edgefkt{e}$ for some $e \in E(H).$ If $H$ is a cycle, this holds true by parity of $X.$
    If $H$ is a \BoB, let $\mathcal{P}$ be any tight set partition in strong correspondence with $\mathcal{T}.$
    By \cref{lem:cutswell}, there is $P \in \mathcal{P}$ such that $X \cap P$ is odd.
    Suppose for a contradiction that $X$ has even intersection with $\edgefkt{e}$ for all $e \in E(H)$ with $\edgefkt{e} \subset P.$
    By \cref{lem:deletemany}, $X \setminus P = X \setminus \bigcup_{e \in E(H): \edgefkt{e} \subset P} \edgefkt{e} \cap X$ is a tight set.
    This contradicts the fact that $X \setminus P$ is even by choice of $P.$
    Thus there exists $e \in E(H)$ such that $X \cap \edgefkt{e}$ is odd.
    This implies via \cref{lem:one_edge_odd} that $\boundary{X}=C$ resides at the edge $e$ in $\mathcal{T}.$
\end{proof}

\begin{corollary} \label{cor:non_pcut_edge}
    Let $\mathcal{T} = \Torsoid{H}$ be a torsoid in $G$ and $X$ a tight set with $0 \neq |\oddIntersections{H}{X}| \neq |V(H)|.$ Then $X$ intersects $\edgefkt{e}$ evenly for any $e \in E(H).$
\end{corollary}

\begin{definition}[\vcut{}]
    \label{def:vcut}
    \label{def:pvcut}
    Let $\mathcal{T}=\Torsoid{H}$ be a torsoid in $G$ and let $v$ be a vertex of $H.$
    Then we say that a tight cut $C$ \emph{resides at a vertex $v$ in $\mathcal{T}$} if there is a tight set $X \subseteq v \cup \bigcup_{v \in e \in \E{H}} \edgefkt{e}$ with $X \cap \edgefkt{e}$ even for all $v \in e \in \E{H}$ and such that $C = \boundary{X}$, and we call $X$ a \emph{\vcut{$\mathcal{T}$} at $v$}.
    We call $X$ a \emph{\pvcut{$\mathcal{T}$}} if $v \subseteq X.$ Then $C$ \emph{resides properly at $v$}.
\end{definition}
See \cref{fig:vertex_resident} for an example of a \BoB with a \vcut{}.

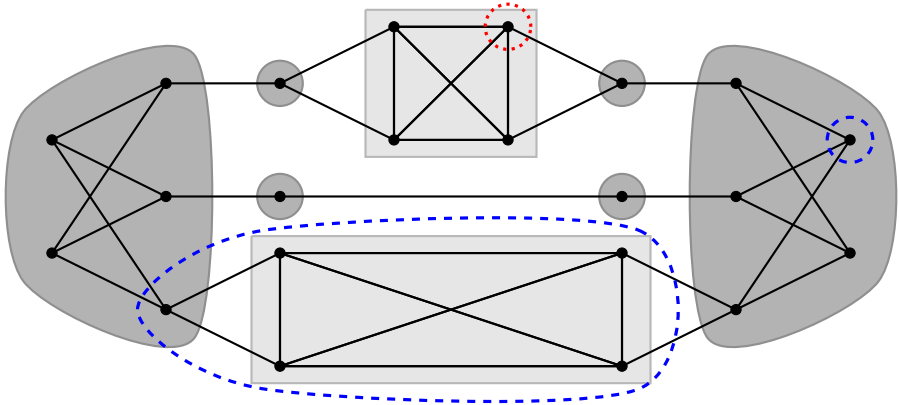
\begin{figure}[ht]
    \centering
    \begin{tikzpicture}[scale=.75]
        \foreach \coordinate/\name in {(6,6)/e1,(8,6)/e2,(8,4)/e3,(6,4)/e4,(4,2)/f1,(10,2)/f2,(10,0)/f3,(4,0)/f4,(0,4)/u1,(2,5)/u2,(2,3)/u3,(2,1)/u4,(0,2)/u5,(4,5)/v,(10,5)/w,(12,5)/x1,(14,4)/x2,(14,2)/x3,(12,1)/x4,(12,3)/x5,(10,3)/y,(4,3)/z} \node[shape=coordinate] at \coordinate (\name) {};
        \foreach \x in {1,5} \foreach \y in {2,3,4} \draw[thick] (u\x) -- (u\y);
        \foreach \x in {2,3} \foreach \y in {1,4,5} \draw[thick] (x\x) -- (x\y);
        \foreach \x in {e,f} \foreach \y in {1,2,3,4} \foreach \z in {1,2,3,4} \draw[thick] (\x\y) -- (\x\z);
        \foreach \x in {u2,e1,e4} \draw[thick] (v) -- (\x);
        \foreach \x in {e2,e3,x1} \draw[thick] (w) -- (\x);
        \foreach \startvertex/\endvertex in {u3/z,x5/y,y/z,f1/u4,f4/u4,f2/x4,f3/x4} \draw[thick] (\startvertex) -- (\endvertex);

        \begin{pgfonlayer}{background}
        \draw [draw=vertexcolor!80!black,thick,fill=vertexcolor] plot [smooth cycle] coordinates {($(u1)+(-0.5,0.5)$) ($(u2)+(0.5,0.5)$) ($(u4)+(0.5,-0.5)$) ($(u5)+(-0.5,-0.5)$)};
        \draw [draw=vertexcolor!80!black,thick,fill=vertexcolor] plot [smooth cycle] coordinates {($(x1)+(-0.5,0.5)$) ($(x2)+(0.5,0.5)$) ($(x3)+(0.5,-0.5)$) ($(x4)+(-0.5,-0.5)$)};
        \draw [draw=edgecolor!80!black,thick,fill=edgecolor] plot coordinates {($(e1)+(-0.5,0.3)$) ($(e2)+(0.5,0.3)$) ($(e3)+(0.5,-0.3)$) ($(e4)+(-0.5,-0.3)$) ($(e1)+(-0.5,0.3)$)};
        \draw [draw=edgecolor!80!black,thick,fill=edgecolor] plot coordinates {($(f1)+(-0.5,0.3)$) ($(f2)+(0.5,0.3)$) ($(f3)+(0.5,-0.3)$) ($(f4)+(-0.5,-0.3)$) ($(f1)+(-0.5,0.3)$)};
        \foreach \vertex in {v,w,y,z} \draw[draw=vertexcolor!80!black,thick,fill=vertexcolor] (\vertex) circle [radius=.4];
        \end{pgfonlayer}
        \draw[very thick, blue, dashed] (x2) circle [radius=.4];
        \draw[very thick, red, dotted] (e2) circle [radius=.4];
        \draw [very thick, blue, dashed] plot [smooth cycle] coordinates {($(f1)+(-0.3,0.4)$) ($(f2)+(0.3,0.4)$) ($(f3)+(0.3,-0.4)$) ($(f4)+(-0.3,-0.4)$) ($(u4)+(-0.5,0)$)};
        \foreach \vertex in {e1,e2,e3,e4,f1,f2,f3,f4,u1,u2,u3,u4,u5,v,w,x1,x2,x3,x4,x5,y,z} \draw[fill] (\vertex) circle [radius=.09];
    \end{tikzpicture}
    \caption{A graph with a possible 6-vertex torsoid $\mathcal{T}$ that is a \BoB and some $\mathcal{T}$-residents.
    The vertex sets enclosed by smooth outlines shaded in \textcolor{vertexcolor!70!black}{darker grey} represent \textcolor{vertexcolor!80!black}{vertices of $\mathcal{T}$} and the two sets enclosed by rectangular outlines filled in \textcolor{edgecolor!70!black}{lighter grey} represent \textcolor{edgecolor!80!black}{edges of $\mathcal{T}$}.
    The vertex sets enclosed by \textcolor{blue}{dashed lines} represent \textcolor{blue}{\vcuts{$\mathcal{T}$}} and the vertex set enclosed by a \textcolor{red}{dotted line} represents a \textcolor{red}{\pcut{$\mathcal{T}$}}.}
    \label{fig:vertex_resident}
\end{figure}

\begin{lemma} \label{lem:vcut_char}
    Let $\mathcal{T}$ be a torsoid in $G.$
    A tight cut $C \subseteq \E{G}$ resides at a vertex of $\mathcal{T}$ if and only if $\theta_{\mathcal{T}}(C) = 1$.
\end{lemma}
\begin{proof}
    By definition, any \vcut{$\mathcal{T}$} $C$ satisfies $\theta_{\mathcal{T}}(C) = 1.$
    For the only if direction let $C$ be a tight cut in $G$ with $\theta_{\mathcal{T}}(C) = 1$, let $X \subseteq \V{G}$ be the tight set with $ \boundary{X} = C$ and $|\V{H}_X| = 1$, and let $v \in V(H)$ be the unique vertex such that $X \cap v$ is odd.
    By \cref{cor:non_pcut_edge}, $X \cap \edgefkt{e}$ is even for any $e \in E(H).$
    It remains to prove that $X$ avoids $w \neq v$ and $\edgefkt{f}$ for $f \in E(H)$ with $v \notin f.$
    Consider an arbitrary tight set partition $\mathcal{P}$ in strong correspondence with $\mathcal{T}$ and let $Q \in \mathcal{P}$ with $v \subseteq Q.$
    As $\oddIntersections{H}{X} = \Set{v}$ and $X$ intersects $\edgefkt{f}$ evenly for $f \in E(H)$, one verifies that $Q$ is the unique element of $\mathcal{P}$ that intersects $X$ oddly.
    Construct a partition $\mathcal{P}'\coloneqq \Set{P \setminus X: v \not\subseteq P \in \mathcal{P}} \cup \Set{Q \cup X}.$
    The partition $\mathcal{P}'$ is indeed a tight set partition by \cref{lem:basictight}.
    Furthermore, $\mathcal{P}$ and $\mathcal{P}'$ correspond.
    By \cref{lem:corrisiso} and \cref{lem:max}, $\collapse{\mathcal{P}'}$ is also \TorsoidInducing.
    Thus we can apply \cref{thm:corriseq}, which shows that, $\mathcal{T}_{\mathcal{P}'} = \mathcal{T}$ holds.
    Thus $Q \cup X \subseteq v \cup \bigcup_{v \in e} \edgefkt{e}.$
    This shows that $X$ is a \vcut{$\mathcal{T}$}, thus $C$ resides at the vertex $v$ in $\mathcal{T}.$
\end{proof}

\begin{lemma}
    Let $\mathcal{T} = \Torsoid{H}$ be a noncyclic torsoid in $G.$
    Any tight cut $C \subseteq \E{G}$ satisfies $\theta_{\mathcal{T}}(C) \leq 1.$
\end{lemma}
\begin{proof}
    Suppose for a contradiction that $C$ is a tight cut in $G$ with $1 < \theta_{\mathcal{T}}(C).$
    Let $X \subseteq \V{G}$ be a vertex set with $\boundary{X}= C$ and $\mathcal{P}$ be any tight cut partition in strong correspondence with $\mathcal{T}.$
    By \cref{cor:non_pcut_edge}, the intersection $X \cap \edgefkt{e}$ is even for any $e \in E(H).$ One verifies that $X$ intersects an element $P \in \mathcal{P}$ oddly if, and only if there is $v \in \oddIntersections{H}{X}$ with $v \subseteq P.$ Thus $|\oddIntersections{\mathcal{P}}{X}| = \theta_{\mathcal{T}}(C).$
    This contradicts \cref{lem:cutswell}, as $H$ is a \BoB.
\end{proof}

The results from this \namecref{sec:relation_tigh_sets_torsoids} so far yield the following corollary:
\begin{corollary}\label{cor:cuts_in_noncyclic_torsoids}
    Let $\mathcal{T}$ be a noncyclic torsoid in $G.$
    Any tight cut of $G$ either resides at a vertex or an edge in $\mathcal{T}.$
    
    Furthermore, for any set $X$ that is tight in $G$ either $X$ or its complement is a \vcut{$\mathcal{T}$} or a \pcut{$\mathcal{T}$}.
    \qed
\end{corollary}

\begin{definition}[\icut{}]
    Let $\mathcal{T}=\Torsoid{H}$ be a cyclic torsoid in $G$ and let $I \subseteq \V{H}$ be an interval in $H$ with $3 \leq |I| \leq |V(H)|-3.$
    Then we say that a tight cut $C$ \emph{resides at the interval $I$ in $\mathcal{T}$} if there is a tight set $X \subseteq \V{G}$ with
    \begin{equation*}
        \bigcup I \cup \bigcup_{vw \in E(H): v,w \in I} \edgefkt{vw} \quad \subseteq \quad X \quad \subseteq \quad \bigcup I \cup \bigcup_{vw \in E(H): v \in I} \edgefkt{vw}
    \end{equation*}
    such that $C = \boundary{X}$ and we call $X$ a \emph{\icut{$\mathcal{T}$} at $I$}.
\end{definition}
See \cref{fig:interval_resident} for an example of a torsoid with an \icut{}.
Note that by \cref{lem:cutscyclicpartition}, any \icut{$\mathcal{T}$} at an interval $I$ intersects any set $\edgefkt{uv}$ evenly for $u \in I$ and $v \notin I.$

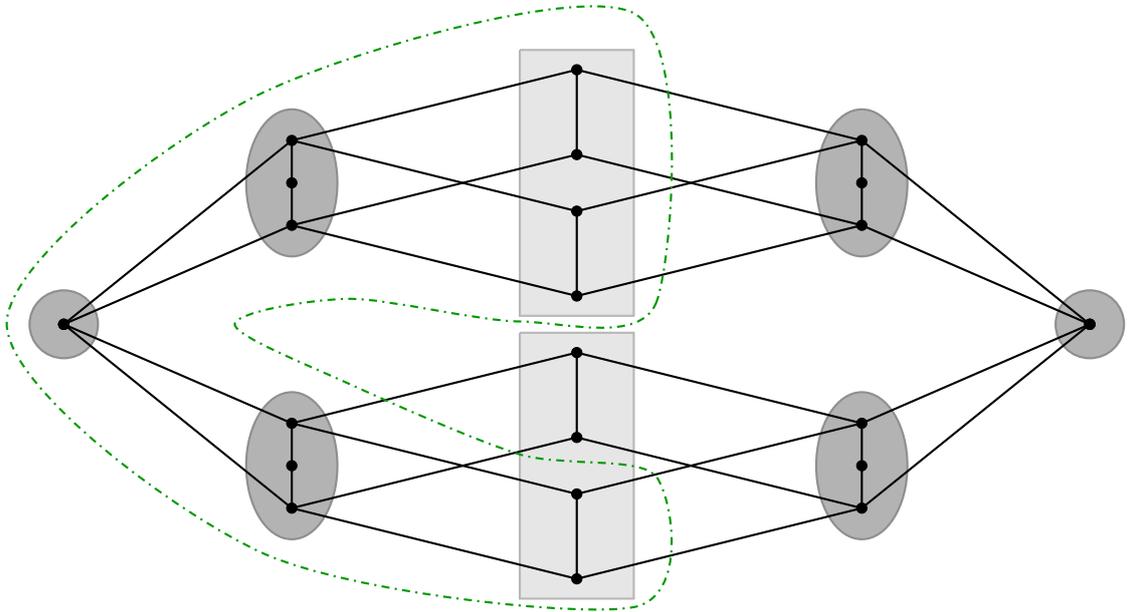
\begin{figure}[ht]
    \centering
    \begin{tikzpicture}[scale=.75]
        \foreach \coordinate/\name in {(0,4)/ao,(0,4)/am,(0,4)/ai,(4,0.75)/bo,(4,1.5)/bm,(4,2.25)/bi,(14,0.75)/co,(14,1.5)/cm,(14,2.25)/ci,(18,4)/do,(18,4)/dm,(18,4)/di,(14,7.25)/eo,(14,6.5)/em,(14,5.75)/ei,(4,7.25)/fo,(4,6.5)/fm,(4,5.75)/fi,(9,3.5)/bc1,(9,2)/bc2,(9,1)/bc3,(9,-0.5)/bc4,(9,8.5)/ef1,(9,7)/ef2,(9,6)/ef3,(9,4.5)/ef4} \node[shape=coordinate] at \coordinate (\name) {};
        \foreach \vertex in {a,d} \draw [draw=vertexcolor!80!black, thick, fill=vertexcolor] (\vertex m) circle [radius=.6];
        \foreach \vertex in {b,c,e,f} \draw [draw=vertexcolor!80!black, thick, fill=vertexcolor] (\vertex m) ellipse (0.8cm and 1.3cm);
        \foreach \group in {ef,bc} \draw [draw=edgecolor!80!black, thick, fill=edgecolor] plot coordinates {($(\group 1)+(-1,0.35)$) ($(\group 1)+(1,0.35)$)  ($(\group 4)+(1,-0.35)$) ($(\group 4)+(-1,-0.35)$) ($(\group 1)+(-1,0.35)$)};
        \foreach \start/\finish in {a/b,c/d,d/e,f/a} \foreach \position in {o,i} \draw[thick] (\start\position) to (\finish\position);
        \foreach \startvertex/\endvertex in {ef1/ef2,ef3/ef4,bc1/bc2,bc3/bc4} \draw[thick] (\startvertex) to (\endvertex);
        \foreach \first in {ef,bc} \foreach \second in {1,2,3,4} \draw[fill] (\first\second) circle [radius=.09];
        \foreach \group in {a,b,c,d,e,f} \draw[thick] (\group o) to (\group m);
        \foreach \group in {a,b,c,d,e,f} \draw[thick] (\group i) to (\group m);
        \foreach \startvertex in {bi,ci} \foreach \endvertex in {bc1,bc3} \draw[thick] (\startvertex) to (\endvertex);
        \foreach \startvertex in {bo,co} \foreach \endvertex in {bc2,bc4} \draw[thick] (\startvertex) to (\endvertex);
        \foreach \startvertex in {ei,fi} \foreach \endvertex in {ef2,ef4} \draw[thick] (\startvertex) to (\endvertex);
        \foreach \startvertex in {eo,fo} \foreach \endvertex in {ef1,ef3} \draw[thick] (\startvertex) to (\endvertex);
        \draw [thick, draw=green!60!black, dash dot] plot [smooth cycle] coordinates { ($(ef1)+(1,1)$) ($(ef4)+(1.4,-0.1)$) ($(ef4)+(-1,-0.45)$) ($(fi)+(1,-1.3)$) ($(am)+(3,0)$) ($(bi)+(1,0.7)$) ($(bc3)+(-1,0.7)$) ($(bc3)+(1.4,0.3)$) ($(bc4)+(1,-0.5)$) ($(bo)+(-0.5,-0.8)$) ($(ao)+(-1,0)$) ($(fo)+(-0.6,0.8)$)};
        \foreach \first in {a,b,c,d,e,f} \foreach \second in {o,m,i} \draw[fill] (\first\second) circle [radius=.09];
    \end{tikzpicture}
    \caption{A graph with a possible torsoid $\mathcal{T}$ that is a $C_6$ and a \icut{$\mathcal{T}$}.
    The vertex sets enclosed by smooth outlines shaded in \textcolor{vertexcolor!80!black}{darker grey} represent \textcolor{vertexcolor!80!black}{vertices of the torsoid $\mathcal{T}$} and the two sets enclosed by rectangular outlines filled in \textcolor{edgecolor!80!black}{lighter grey} represent \textcolor{edgecolor!80!black}{edges of $\mathcal{T}$}.
    The vertex set enclosed by the \textcolor{green!60!black}{dash dotted line} represents a \textcolor{green!60!black}{\icut{$\mathcal{T}$}}.}
    \label{fig:interval_resident}
\end{figure}

\begin{lemma}
    Let $\mathcal{T} = \Torsoid{H}$ be a cyclic torsoid in $G$ and $C \subseteq \E{G}$ any tight set. The set $C$ resides at an interval in $\mathcal{T}$ if and only if $3 \leq \theta_{\mathcal{T}}(C)$ holds.
\end{lemma}

\begin{proof}
    By definition, any tight cut $C$ that resides at an interval in $\mathcal{T}$ satisfies $3 \leq \theta_{\mathcal{T}}(C) .$
    For the only if direction let $C$ be a tight cut with $3 \leq \theta_{\mathcal{T}}(C)$, let $X \subseteq \V{G}$ be the tight set with $\boundary{X} = C$, $\oddIntersections{H}{X} = \theta_{\mathcal{T}}(C)$ and let $\mathcal{P}$ be any tight set partition in strong correspondence with $\mathcal{T}.$
    Set $n\coloneqq |\mathcal{P}| = |V(H)|$ and $m\coloneqq |\oddIntersections{\mathcal{P}}{X}|.$
    We apply \cref{lem:cutscyclicpartition} to obtain a cyclic enumeration $P_1, \dots, P_{n}$ of $\mathcal{P}$ such that $\oddIntersections{\mathcal{P}}{X} = \Set{P_1, \dots, P_{m}}.$
    Let $v_i \in V(H)$ with $v_i \subseteq P_i$ for $i \in [n].$
    We assumed $\theta_{\mathcal{T}}(C) \neq 0$, this 
    implies that $|\oddIntersections{H}{X}| \notin \Set{0, |V(H)|}$,
    thus $X$ has even intersection with $\edgefkt{f}$ for any $f \in E(H)$ by \cref{cor:non_pcut_edge}.
    Thus for $i \in [n]$ we have $v_i \in \V{H}_X$ if and only if $P_i \in \oddIntersections{\mathcal{P}}{X}.$
    This implies $m = \theta_{\mathcal{T}}(C)$ and we can deduce by \cref{lem:cutscyclicpartition} that $X \setminus \bigcup \oddIntersections{\mathcal{P}}{X}$ can be partitioned into $S', S''$ such that $S'$ is passable between $P_{n}, P_1$ and $S''$ is passable between $P_{m}, P_{m + 1}.$
    By \cref{lem:modifypass}, $S'$ is passable between $v_n$ and $v_1.$
    Thus it is contained in $\edgefkt{v_n v_1}.$
    Similarly $S''$ is contained in $\edgefkt{v_m, v_{m +1}}.$ By the same argument $\bigcup \oddIntersections{\mathcal{P}}{X} \setminus X$ can be partitioned into passable sets between $v_n, v_1$ and $v_m, v_{m+1}.$
    This completes the proof.
\end{proof}
Note that for any tight cut $C$ the number $\theta_{\mathcal{T}}(C)$ is either $0$ or odd.
\begin{corollary}\label{cor:cuts_in_general_torsoids}
    Let $\mathcal{T}$ be a cyclic torsoid in $G.$
    Any tight cut of $G$ either resides at a vertex, an edge or an interval in $\mathcal{T}.$
    
    Furthermore, for any tight set $X$ of $G$ either $X$ or its complement is a \vcut{$\mathcal{T}$}, a \pcut{$\mathcal{T}$}, or a \icut{$\mathcal{T}$}.
    \qed
\end{corollary}

\subsection{Correspondences between torsoids and tight set partitions}
Using the results of the last subsection, we can show that not only strong correspondences, but arbitrary correspondences between torsoids and tight set partitions have some good properties.

\begin{lemma}\label{lem:epspass}
    Let $\mathcal{P}$ be a tight set partition of $G$ such that either $\collapse{\mathcal{P}}$ is a \BoB\ or $\mathcal{P}$ is a maximal cycle.
    Let $\sigma$ be a correspondence from a torsoid $\mathcal{T}= \Torsoid{H}$ in $G$ to $\mathcal{P}$ and let $vw$ be an edge of $H.$
    Then $\edgefkt{vw}$ is the largest set passable between $\sigma(v)$ and $\sigma(w).$ 
\end{lemma}
\begin{proof}
    Let $P = \sigma(v)$ and $Q = \sigma(w).$
    We begin by showing that $\edgefkt{vw}$ is passable between $P$ and $Q.$
    By symmetry it is enough to show that it is passable for $P.$
    By \cref{cor:cuts_in_noncyclic_torsoids} and \cref{cor:cuts_in_general_torsoids}, $P$ must be a vertex cut with respect to $(H,\edgefkt{}).$
    Thus $P \setminus \edgefkt{vw}$ is tight by \cref{lem:basictight} applied to $P$ and the complement of $w \cup \edgefkt{vw}$, and $P \cup \edgefkt{vw}$ is tight by \cref{lem:basictight} applied to $P$ and $v \cup \edgefkt{vw}.$ 
    
    Now let $S$ be the largest set passable between $P$ and $Q.$
    From the arguments in the last paragraph we already know that for any neighbour $x$ of $v$ other than $w$ in $H$ the set $\edgefkt{vx}$ is passable for $\mathcal{P}.$
    So by \cref{lem:disjointpass} and \cref{lem:modifypass} the set $S$ is passable for $P \setminus \bigcup_{x \in N_H(v) \setminus \Set{w}} \edgefkt{vx}.$
    Since $\edgefkt{vw} \subseteq S$ this implies that $S$ is also passable for $P \setminus \bigcup_{x \in N_H(v)} \edgefkt{vx}.$
    This latter set is just equal to $v$ because $v \subseteq P \subseteq v \cup \bigcup_{x \in N_H(v)} \edgefkt{vx}.$
    
    A similar argument shows that $S$ is passable for $w$, and by \cref{lem:disjointpass} it is a subset of $v \cup \edgefkt{vw} \cup w$, so we have $S \subseteq \edgefkt{vw}.$ Since, as we saw in the previous paragraph, $\edgefkt{vw}$ is passable between $P$ and $Q$, we also have $\edgefkt{vw} \subseteq S.$
    Thus $S = \edgefkt{vw}$, as required.
\end{proof}

\begin{theorem}\label{thm:uniquecorr}
    Let $\mathcal{P}$ be a \TorsoidInducing tight set partition of $G.$
    Then the only torsoid $\mathcal{T}$ in correspondence with $\mathcal{P}$ is $\mathcal{T}_{\mathcal{P}}$ and the only such correspondence is $\sigma_{\mathcal{P}}.$
\end{theorem}

\begin{proof}
    Let $\sigma$ be any correspondence from any torsoid $\mathcal{T} = (H, \edgefkt{})$ to $\mathcal{P}.$
    By \cref{lem:epspass}, for any edge $vw$ of $H$ we have $\edgefkt{vw} = \delta_{\mathcal{P}}(\sigma(v)\sigma(w)).$
    For any vertex $v$ of $H$, since $v \subseteq \sigma(v) \subseteq v \cup \bigcup_{w \in N_H(v)}\edgefkt{vw}$ we have 
    \begin{equation*}
        v = \sigma(v) \setminus \bigcup_{w \in N_H(v)} \edgefkt{vw} = \sigma(v) \setminus \bigcup_{Q \in N_{\collapse{\mathcal{P}}}(\sigma(v))}\delta_{\mathcal{P}}(\sigma(v)Q) = \tau_{\mathcal{P}}(\sigma(v))\,.
    \end{equation*}
    
    Thus $\tau_{\mathcal{P}} = \sigma^{-1}$ and so $\sigma_{\mathcal{P}} = \sigma.$
    It follows that $H_{\mathcal{P}} = H$ and $\varepsilon_{\mathcal{P}} = \edgefkt{}$, giving the desired result.
\end{proof}

\begin{corollary}\label{cor:corrisiso}
    If $\sigma$ is a correspondence from a torsoid $\mathcal{T} =(H,\edgefkt{})$ in $G$ to a tight set partition $\mathcal{P}$ of $G$ then it is a graph isomorphism between $H$ and $\collapse{\mathcal{P}}.$
    \qed
\end{corollary}

\section{Relation of torsos to torsoids}
\label{sec:torsos_to_torsoids}

In this \namecref{sec:torsos_to_torsoids} we consider specific tight cut contractions of a graph with respect to a fixed maximal family of nested tight cuts and investigate how they relate to our concept of torsoids.
The results of this \namecref{sec:torsos_to_torsoids} emphasise that torsoids are a global tool capturing all of these tight cut contractions independently of the precise choice of a tight cut family.

\begin{definition}[torso]
    Let $\mathcal{C}$ be a maximal family of nested tight cuts in $G.$
    A \emph{maximal star} of $\mathcal{C}$ is a tight set partition $\Partition$ of size at least $4$ such that $\collapse{\Partition}$ is a \BoB and $\boundary{P} \in \mathcal{C}$ for every $P \in \Partition.$
    Then we call $\collapse{\Partition}$ a \emph{torso of $\mathcal{C}$ at the maximal star $\mathcal{\Partition}$}, or a \emph{torso} for short.
    
    Furthermore, we simply call a torso at some maximal star of some maximal family of nested tight cuts in $G$ a \emph{torso in $G$}.
\end{definition}
If a torso is a $C_4$, we call it a \emph{$C_4$-torso}.
Otherwise, the torso is a \BoB other than $C_4$ and we call it a \emph{non-$C_4$-torso}.

We want to emphasise the significant difference between torsoids and torsos.
Torsos are completely dependent on a fixed family of tight cuts $\mathcal{C}$, they capture how this specific family behaves and can only be used to describe the properties of $\mathcal{C}.$
This is different for torsoids, a torsoid is a global object that captures multiple ways to choose maximal families of nested tight cuts by keeping passable sets on its edges rather than assigning them rigidly and by bundling any maximal cyclic structure in one torsoid instead of considering it as different torsos.

The relation of torsos and torsoids that we investigate in this section is defined as follows:
\begin{definition}[cleave]
    Let $\mathcal{S}$ be a torso in $G$ and $\mathcal{T}$ a torsoid in $G.$
    We say $\mathcal{S}$ \emph{cleaves} $\mathcal{T}$ if every vertex of $\mathcal{S}$ contains a vertex of $\mathcal{T}.$
\end{definition}

In \cref{subsec:every_torso_a_unique_torsoid} we prove that each torso cleaves exactly one torsoid.
However, in the other direction there might be several torsos that cleave a given torsoid, as cyclic torsoids always induce maximal cyclic tight set partitions in contrast to cyclic torsos whose induced tight set partition can be non-maximal cyclic.
So, in \cref{subsec:torso_cleaving_torsoid} we prove the main result of this \namecref{sec:torsos_to_torsoids} that describes how many torsos cleave a given torsoid.

\begin{observation}
	\label{obs:refining_partition}
Note that all vertices of a torso $\mathcal{S}$ cleaving a torsoid $\mathcal{T}$ are either \pvcuts{$\mathcal{T}$} or \icuts{$\mathcal{T}$}.
Thus every vertex of $\mathcal{T}$ is contained in a vertex of $\mathcal{S}.$
Therefore there exists a partition in correspondence to $\mathcal{T}$ that refines the partition $V(\mathcal{S}).$
\end{observation}

\subsection{Every torso cleaves a unique torsoid}\label{subsec:every_torso_a_unique_torsoid}

\begin{lemma}\label{lem:cyclicnoncyclic}
    Let $\mathcal{S}$ be a torso in $G$ that cleaves a torsoid $\mathcal{T}$ in $G.$
    Then $\mathcal{T}$ is cyclic if and only if $\mathcal{S}$ is a $C_4.$
\end{lemma}
\begin{proof}
    If $\mathcal{T}$ is a cyclic torsoid, then every vertex of $\mathcal{S}$ is either a \pvcut{$\mathcal{T}$}~or a \icut{$\mathcal{T}$}.
    Then $\mathcal{S}$ is a cycle.
    As $\mathcal{S}$ is a \BoB, it is a $C_4.$
    
    If $\mathcal{T}$ is a noncyclic torsoid, then every vertex of $\mathcal{S}'$ is a \pvcut{$\mathcal{T}'$}.
    This implies that the partition $\V{\mathcal{S}'}$ corresponds to $\mathcal{T}.$
    Therefore $\mathcal{S}'$ is not a cycle and thus a non-$C_4$-torso.
\end{proof}

\begin{lemma}
    Let $\mathcal{S}$ be a torso in $G.$
    There exists exactly one torsoid $\mathcal{T}$ in $G$ such that $\mathcal{S}$ cleaves $\mathcal{T}.$
\end{lemma}
\begin{proof}
    If $\mathcal{S}$ is non-$C_4$, it is a \BoB~other than $C_4$ and thus by \cref{def:inducedtorsoid,thm:inducedtorsoid}, $\mathcal{T}_{\V{\mathcal{S}}}$ is a torsoid.
    By construction, $\mathcal{S}$ cleaves $\mathcal{T}_{\V{\mathcal{S}}}.$
    Furthermore, any torsoid $\mathcal{T}$ cleaved by $\mathcal{S}$ is noncyclic by \cref{lem:cyclicnoncyclic}.
    Thus any vertex of $\mathcal{S}$ is a \pvcut{$\mathcal{T}$} and thus $\mathcal{T} = \mathcal{T}_{\V{\mathcal{S}}}$, since $\mathcal{T}$ and $\V{\mathcal{S}}$ correspond.
    Thus we can assume that $\mathcal{S}$ is a $C_4$-torso.
    
    We use refinement to show that there exists a torsoid that $\mathcal{S}$ cleaves:
    Given any $C_4$-torso and its underlying tight set partition $\Partition_1\coloneqq \Partition$, we begin by, if possible, refining a partition class $P \in \Partition_1 $ in such a way that we split it up into three smaller sets to obtain a new partition $\Partition_2$ such that $\collapse{\Partition_2}$ is again a cycle.
    Note that if $P$ can be split up into finitely many smaller sets it can always be split up into exactly three, since in a cyclic partition any odd union of consecutive partition classes is a tight cut.
    We repeat this as long as possible.
    \begin{claim}
        \label{claim:finite_process}
        The process stops after finitely many refinement steps.
    \end{claim}
    \begin{claimproof}
        Suppose for a contradiction that infinitely many refining steps are possible.
    Then there are partitions $(\Partition_i)_{i \in \N}$ and tight sets $ P_1 \supset P_2 \supset P_3 \supset \dots $ where $P_i \in \Partition_i$, since any $\mathcal{P}_i$ is finite.
    Without loss of generality we may assume that $\Partition_{i}$ is a refinement of $\Partition_{i-1}$ in such a way that only $P_{i - 1}$ is refined (otherwise obtain such a partition by unifying some consecutive tight sets of $\Partition_{i}$).
    Furthermore we may assume that for each of the $\Partition_i$ we fixed a circular ordering of the partition sets in such a way that the order on $\Partition_i$ induces the one on $\Partition_{i-1}.$
    
    Let $Q$ be a partition set of $\Partition_1$ other than $P_1$ or one of its neighbours in $\collapse{\Partition_1}.$
    By assumption, for each $i \in \N$ we have $Q \in \Partition_i$ and furthermore the circular ordering on $\Partition_i$ induces a linear ordering on $\Partition_i \setminus \Set{Q}.$
    For $\sigma\in \set{-,+}$ we define $P_i^\sigma$ to be the union of all sets in $\Partition_i \setminus \Set{Q}$ that are $\set{\text{smaller},\text{larger}}$ than $P_i$ in this linear order.
    By definition it holds $P^\sigma_i \subseteq P^\sigma_{i+1}$ and we define $P_\infty^\sigma \coloneqq \bigcup_{i \in \N} P_i^\sigma $ for $\sigma \in \set{+,-}.$
    Note that this definition implies that $P_\infty^+$ and $P_\infty^-$ are disjoint.
    Next, it is clear that $\Abs{\mathcal{P}_i} \geq 4$, thus we may pick a path $R$ from $P_1^-$ to $P_1^+$ in $G$ that is disjoint from $Q.$
    For one $\sigma \in \set{+,-}$ we have $P_i^\sigma \subsetneq P_{i+1}^\sigma$ for infinitely many distinct $i \in \N$, without loss of generality let $\sigma=-.$
    Then there is an edge $vw$ of $R$ with $v \in P_\infty^-$ and $w \not\in P_\infty^-.$
    Furthermore, there are natural numbers $k<\ell$ such that $k$ is the smallest number with  $v \in P_k^-$ and $P_k^- \subsetneq P_\ell^-.$
    Therefore $vw$ witnesses that there is a chord in $\collapse{\mathcal{P}_\ell}$, a contradiction to the assumption that $\collapse{\mathcal{P}_\ell}$ is a cycle.
    Thus the refinement stops at some $n \in \N$, as desired.
    \end{claimproof}
    By \cref{claim:finite_process} there is $n \in \N$ and a sequence $(\mathcal{P}_i)_{i \in [n]}$ as described above such that $\Partition_n$ is maximal cyclic.
    Thus $\Partition_n$ induces a torsoid $\mathcal{T}_{\mathcal{P}_n}$ by \cref{thm:inducedtorsoid}, which is cleaved by $\mathcal{S}.$

    Let $\mathcal{T}$ be any torsoid that is cleaved by $\mathcal{S}.$
    By \cref{lem:cyclicnoncyclic}, the torsoid $\mathcal{T}$ is cyclic.
    We show:
    \begin{claim}
        \label{claim:every_partition_contains_a_vertex_of_T}
        For each $i \in [n]$ every $P \in \Partition_i$ contains a vertex of $\mathcal{T}.$
    \end{claim}
    \Cref{claim:every_partition_contains_a_vertex_of_T} implies that every $P \in \mathcal{P}_i$ is either a \icut{$\mathcal{T}$} or a \pvcut{$\mathcal{T}$} and there exists a partition in correspondence with $\mathcal{T}$ that refines $\mathcal{P}_i.$
    Let $\mathcal{Q}$ be the partition in correspondence with $\mathcal{T}$ that refines $\mathcal{P}_n.$
    Since $\mathcal{P}_n$ is maximal cyclic and $\mathcal{Q}$ is cyclic, $\mathcal{Q}$ coincides with $\mathcal{P}_n$ and thus $\mathcal{T} = \mathcal{T}_{\mathcal{P}_n}$ holds. Therefore $\mathcal{T}_{\mathcal{P}_n}$ is the unique torsoid cleaved by $\mathcal{S}.$

    \begin{claimproof}
    Proof by induction on $i \in [n].$ By construction, $\mathcal{P} = \mathcal{P}_1$ has the desired property.
    Suppose that any partition class of $\mathcal{P}_{i -1}$ contains a vertex of $\mathcal{T}.$
    We prove that also $\mathcal{P}_i$ has this property.
    
    Let $P \in \mathcal{P}_{i -1}$ and $P_1, P_2, P_3 \in \mathcal{P}_i$ in cyclic order such that $P = P_1 \disjointUnion P_2 \disjointUnion P_3$ with $\mathcal{P}_i = \Set{P_1, P_2, P_3} \cup \Brace{\mathcal{P}_{i - 1} \setminus \Set{P}}.$
    By \cref{obs:refining_partition}, there exists a tight set partition $\mathcal{R}$ in correspondence to $\mathcal{T}$ that refines $\mathcal{P}_{i-1}.$
    We define $\mathcal{P}_i'\coloneqq \Set{P} \cup \Set{R \in \mathcal{R}: R \subset \V{G} \setminus P}.$
    As $\mathcal{R}$ is a refinement of $\mathcal{P}_i'$ and $\collapse{\mathcal{R}}$ is cyclic, also $\collapse{\mathcal{P}_i'}$ is cyclic.
    We consider $\mathcal{P}_i''\coloneqq \Set{P_1, P_2, P_3} \cup \Set{R \in \mathcal{R}: R \subset \V{G} \setminus P}.$
    By construction, $\mathcal{P}_i''$ is a refinement of $\mathcal{P}_i'$ and $\mathcal{P}_i.$

    We prove that $\mathcal{P}_i''$ is also cyclic.
    Let $\hat{P}, \tilde{P}$ be the neighbours of $P$ in $\mathcal{P}_{i - 1}.$
    Let $\hat{R} \subseteq \hat{P}$ and $\tilde{R} \subseteq \tilde{P}$ be the neighbours of $P$ in $\mathcal{P}_i'.$
    Without loss of generality, $\cut{}{P}= \E{P_1, \hat{P}} \disjointUnion \E{P_3, \tilde{P}}$, since $\mathcal{P}_i$ is cyclic.
    Then $\cut{}{P} = \E{P, \hat{R}} \disjointUnion \E{P, \Tilde{R}}$, since $\mathcal{P}_i'$ is cyclic.
    This implies
    \begin{equation*}
        \cut{}{P} = \Brace{\E{P, \hat{R}} \cap \E{P_1, \hat{P}}} \disjointUnion \Brace{\E{P, \tilde{R}} \cap \E{P_3, \tilde{P}}} = \E{P_1, \hat{R}} \disjointUnion \E{P_3, \hat{P}},
    \end{equation*}
    which proves that $\mathcal{P}_i''$ is cyclic.
    
    If $P$ is a \pvcut{$\mathcal{T}$}, then $\mathcal{P}_i'$ corresponds to $\mathcal{T}.$
    Therefore $\mathcal{P}_i'$ is maximal cyclic, which contradicts that $\mathcal{P}_i''$ is cyclic and refines $\mathcal{P}_i'.$
    Thus $P$ is a \icut{$\mathcal{T}$}.
    In particular, $P$ contains at least $3$ vertices of $\mathcal{T}.$
    Suppose towards a contradiction that one of $P_1, P_2, P_3$ does not contain a vertex of $\mathcal{T}$, i.e.~it is either a \npvcut{$\mathcal{T}$} or a \pcut{$\mathcal{T}$}.
    If it is a \npvcut{$\mathcal{T}$}, then this \npvcut{$\mathcal{T}$} contains a proper, odd subset of a vertex of $\mathcal{T}.$ Since \pcuts{$\mathcal{T}$} and \icuts{$\mathcal{T}$} either are supersets of this vertex or avoid this vertex,
    all three tight sets $P_1, P_2, P_3$ are \vcuts{$\mathcal{T}$} at that same vertex. Thus $P_1, P_2, P_3$ do not contain any other vertex of $\mathcal{T}$, a contradiction as $P = P_1 \cup P_2 \cup P_3.$
    
    Thus one of $P_1, P_2, P_3$ is a \pcut{$\mathcal{T}$}.
    By parity, two of $P_1, P_2, P_3$ are \pcuts{$\mathcal{T}$} at the same edge.
    The other tight set is a \icut{$\mathcal{T}$}, since $P$ contains at least $3$ vertices of $\mathcal{T}.$
    The boundary of this \icut{} contains an edge of $\cut{}{P}$, as the two other tight sets are \pcut{} at the same edge.
    Thus either $P_1$ or $P_3$ is the \icut{} and we suppose without loss of generality the former.
    
    We prove that $\mathcal{R}$ is not maximal cyclic, a contradiction.
    Note that $P_2, P_3 \subset P$ are \pcuts{$\mathcal{T}$} at an edge $uw \in \E{H}$ with $u \subset P$ and $w \cap P = \emptyset.$
    Therefore, both are contained in the element $R \in \mathcal{R}$ that contains $u.$
    We show that $\mathcal{R}^* = \{R \cap P_1, P_2, P_3\} \cup \mathcal{R} \setminus \Set{R}$ is a tight set partition such that $\collapse{\mathcal{R}^*}$ is cyclic.
    By construction, $E(P_1, P_2) = E(P_1 \cap R, P_2)$ holds.
    Then since $\mathcal{P}_i''$ and $\mathcal{R}$ are cyclic, also $\mathcal{R}^*$ is cyclic, which provides the desired contradiction.
    \end{claimproof}
    This completes the proof.
\end{proof}

\subsection{Torsos cleaving a fixed torsoid}\label{subsec:torso_cleaving_torsoid}

As we have established that every torso cleaves exactly one torsoid we can introduce the following function.

\begin{definition}[$\kappa$]
    Let $\mathcal{C}$ be a maximal family of nested tight cuts in $G$, $\mathfrak{S}$ be the set of all torsos of $\mathcal{C}$ and $\mathfrak{T}$ the set of all torsoids in $G.$
    We define the map $\kappa_{\mathcal{C}} \colon \mathfrak{S} \longrightarrow \mathfrak{T}$ sending $\mathcal{S} \in \mathfrak{S}$ to the unique torsoid $\mathcal{T}$ such that $\mathcal{S}$ cleaves $\mathcal{T}.$
\end{definition}

Now we want to know what we can say about the torsos that cleave a given torsoid.
This question leads to the main result of this \namecref{sec:torsos_to_torsoids}, which reads as follows:

\begin{restatable}{theorem}{correspondTorsoTorsoidTheorem}
	\label{thm:correspond_torso_torsoid}
    Let $\mathcal{C}$ be a maximal family of nested tight cuts in $G.$
    Then the $\kappa_{\mathcal{C}}$-pre-image of a noncyclic torsoid consists of one non-$C_4$ torso and the $\kappa_{\mathcal{C}}$-pre-image of a cyclic torsoid with $n$ vertices consists of $\tfrac{n}{2} - 1$ $C_4$-torsos.
\end{restatable}

At first we consider the special case of \cref{thm:correspond_torso_torsoid} for \BoBs and cycles.
Note that every \BoB and every cycle contains a unique torsoid.
Furthermore, any torso in a \BoB or in a cycle cleaves this unique torsoid.

\begin{proposition}
\label{prop:number_torso}
    Let $H$ be a matching covered graph with a maximal family $\mathcal{C}$ of nested tight cuts in $H.$
    If $H$ is a \BoB, there is exactly one torso of $\mathcal{C}.$
    If $H$ is a cycle of even length $n$, there are exactly $\frac{n}{2} - 1$ torsos of $\mathcal{C}.$
\end{proposition}
\begin{proof}
    If $H$ is a \BoB, the set $\mathcal{C}$ contains only trivial tight sets and thus $H$ itself is the unique torso of $\mathcal{C}.$
    If $H$ is a cycle, we prove the statement by strong induction on $n.$
    For $n = 4$ the statement is true, as $C_4$ is a \BoB.
    Let $H$ be a cycle of length $n \geq 6$ and suppose the statement is true for all cycles of length at most $n - 2.$
    Let $\cut{}{X} \in \mathcal{C}$ be any non-trivial tight cut.
    Set $\mathcal{C}_1 \coloneqq \Set{C' \in \mathcal{C}: \exists Y \subseteq \V{G} \setminus X: C' = \cut{}{Y}}$ and $\mathcal{C}_2 \coloneqq \Set{C' \in \mathcal{C}: \exists Y \subseteq X: C' = \cut{}{Y}}.$
    Note that $\mathcal{C} = \mathcal{C}_1 \cup \mathcal{C}_2$ and $\Set{\cut{}{X}} = \mathcal{C}_1 \cap \mathcal{C}_2.$
    
    Let $H_1$ be the cycle obtained by contracting $X$ and consider the set $\hat{\mathcal{C}_1}$ of nested tight cuts induced by $\mathcal{C}_1.$
    It is maximal by construction.
    Let $H_2$ be the cycle obtained by contracting $\V{H} \setminus X$ and consider the the set $\hat{\mathcal{C}_2}$ of nested tight cuts induced by $\mathcal{C}_2$, which is also maximal.
    There is a canonical bijection between the set of torso of $\mathcal{C}$ and the set of torsos of $\hat{\mathcal{C}_1}$ and $\hat{\mathcal{C}_2}.$
    Applying the induction hypothesis to $\mathcal{C}_1, H_1$ and $\mathcal{C}_2, H_2$ gives the desired statement for $n.$
\end{proof}

In preparation of the proof for the general case of \cref{thm:correspond_torso_torsoid} we prove a couple of propositions that enable us to define a bijection between the $\kappa_{\mathcal{C}}$-pre-image of a torsoid $\mathcal{T}=\Torsoid{}$ and the set of torsos of a specific maximal family of nested tight cuts in $H$. Note that \cref{prop:number_torso} determines the size of the latter set, as $H$ is either a \BoB or a cycle.

Firstly, we show that the vertices of a torso cleaving $\mathcal{T}$ are $(\mathcal{T}, \tightCutsToSets{\mathcal{C}})$-\maxcut which we define below. Secondly, we introduce the specific family of nested tight cuts in $H$ and show that it is indeed maximal.

Now we return to the graph $G$ itself and consider the residents on vertices and intervals of a torsoid $\mathcal{T} = (H,\epsilon)$ in $G.$
Recall that $\oddIntersections{H}{X}$ is the set of vertices in $H$ that intersect $X$ oddly.
We call a \pvcut{$\mathcal{T}$} or an \icut{$\mathcal{T}$} $X \in \tightCutsToSets{\mathcal{C}}$ \emph{$(\mathcal{T}, \tightCutsToSets{\mathcal{C}})$-\maxcut}, if for every tight set $Y \in \tightCutsToSets{\mathcal{C}}$ with $\oddIntersections{H}{Y} = \oddIntersections{H}{X}$ holds $Y \subseteq X.$
Note that $\bigcup \oddIntersections{H}{X} \subseteq X$ holds for every \pvcut{$\mathcal{T}$} or \icut{$\mathcal{T}$} $X.$

The next two statements prove that for every vertex and every interval that has a resident such maximal resident exists.

\begin{proposition}\label{prop:maxcut_vertex}
    Let $\mathcal{C}$ be a maximal family of nested tight cuts in $G$ and $\mathcal{T}=(H,\epsilon)$ a torsoid in $G.$
    Then for every $v \in V(H)$ there is a $(\mathcal{T}, \tightCutsToSets{\mathcal{C}})$-\maxcut \pvcut{$\mathcal{T}$} at $v.$
\end{proposition}
\begin{proof}
    Consider the set $A\coloneqq\Set{Y \in \tightCutsToSets{\mathcal{C}}: \oddIntersections{H}{Y}=\Set{v}}.$
    We show that $X\coloneqq \bigcup A$ is the desired set.
    By construction, the set $A$ contains only \vcuts{$\mathcal{T}$}.
    Thus $X$ avoids all vertices of $\mathcal{T}$ apart from $v.$
    Furthermore, the set $X$ contains the set $v$ as every single vertex in $v$ forms a \vcut{$\mathcal{T}$}, whose cut is clearly nested with $\mathcal{C}.$
    Therefore $\oddIntersections{H}{X} = \Set{v}$ holds.
    Thus $X$ is a \pvcut{$\mathcal{T}$}.
    We have to prove that $X$ is tight and an element of $\tightCutsToSets{\mathcal{C}}.$
    
    As $v \subseteq X$, $X = \bigcup_{Y \in A} \Brace{Y \cup v}$ holds.
    For all $Y', Y'' \in A$ the sets $Y' \cup v, Y'' \cup v$ are tight.
    Since $Y', Y''$ are nested, either the sets $Y', Y''$ are disjoint or one contains the other.
    In the former case, $(Y' \cup v) \cap (Y'' \cup v) = v$ is odd.
    In the latter case, $(Y' \cup v) \cap (Y'' \cup v) \in \Set{Y' \cup v, Y'' \cup v}$, which is also odd.
    In either case \cref{lem:basictight} implies that the union of the two is a tight set.
    Then, by \cref{lem:infunion}, $X$ is tight.

    It remains to show that $\cut{}{X}$ is nested with every tight cut in $\mathcal{C}.$
    Let $C$ be an arbitrary tight cut in $\mathcal{C}.$
    Note that $C$ is nested with every tight cut induced by an element of $A.$
    If $C$ resides at $v$, there is $Y \in A$ such that $C = \cut{}{Y}$ and thus $\cut{}{X}$ and $C$ are nested by construction.
    If $C$ resides at another vertex, its resident avoids all elements of $A$ and therefore $\cut{}{X}$ and $C$ are nested.
    If $C$ resides at an edge, let $Z$ be the \pcut{} of $C.$
    Either $Z$ is contained in an element of $A$, or $Z$ is disjoint to all elements of $A.$
    In the former case $Z \subseteq X$ holds, in the latter case $Z \cap X = \emptyset$ holds.
    In both cases $\cut{}{X}$ and $C$ are nested.
    If $C$ resides at an interval, let $Z$ be the \icut{} of $C$ containing $v.$
    For any $Y \in A$, the tight cuts $\cut{}{Y}$ and $C = \cut{}{Z}$ are nested and $Y, Z$ have nonempty intersection.
    Therefore $Y$ is contained in $Z.$
    Thus $X$ is contained in $Z$ and $\cut{}{X}$ and $C$ are nested.
\end{proof}

\begin{proposition}\label{prop:maxcut_interval}
    Let $\mathcal{C}$ be a maximal family of nested tight cuts in $G$ and $\mathcal{T}=\Torsoid{H}$ be a cyclic torsoid in $G.$ Furthermore, let $I$ be an interval in $H$ such that there exists $Y \in \tightCutsToSets{\mathcal{C}}$ with $\oddIntersections{H}{Y} = I$ .
    Then there is a $(\mathcal{T}, \tightCutsToSets{\mathcal{C}})$-\maxcut \icut{$\mathcal{T}$} $X$ with $\oddIntersections{H}{X} = I.$
\end{proposition}
\begin{proof}
    Consider the set $A\coloneqq\Set{Y \in \tightCutsToSets{\mathcal{C}}: \oddIntersections{H}{Y}= I}.$
    We prove that $X\coloneqq \bigcup A$ is the desired set.
    As $A$ contains only \icuts{}, the set $X$ contains the vertices in $I$ and avoids all other vertices of $\mathcal{T}.$
    Therefore it holds $\oddIntersections{H}{X} = I.$
    For any two elements $Y', Y'' \in A$ one is contained in the other, since $\bigcup I \subseteq Y', Y''.$
    Therefore $Y', Y''$ intersect oddly and by \cref{lem:infunion} $X$ is tight.
    To complete the proof we have to show that the set $X$ is contained in $\tightCutsToSets{\mathcal{C}}.$
    
    Let $\cut{}{Z} \in \mathcal{C}$ be an arbitrary tight cut.
    It suffices to prove that $\cut{}{Z}$ and $\cut{}{X}$ are nested.
    If $Z$ is contained in an element of $A$, it is contained in $X$ and thus $\cut{}{X}$ and $\cut{}{Z}$ are nested.
    We suppose that $Z$ is not contained in any element of $A.$
    If $Z$ is disjoint to all elements of $A$, the sets $Z$ and $X$ are disjoint and thus $\cut{}{X}$ and $\cut{}{Z}$ are nested.
    Therefore we suppose that there is a $Y \in A$ such that $Y$ and $Z$ have nonempty intersection.
    As $\cut{}{Z}$ and $\cut{}{Y}$ are nested, $Y \subset Z$ holds by assumption.
    This implies $\bigcup I \subset Z.$
    Therefore any $Y' \in A$ intersects $Z$ and as $\cut{}{Y'}, \cut{}{Z}$ are nested, $Y'$ is contained in $Z$ by assumption.
    Thus $X \subseteq Z$, which shows that $\cut{}{X}$ and $\cut{}{Z}$ are nested.
\end{proof}

Next we show that the vertices of a torso cleaving $\mathcal{T}$ are indeed such maximal residents.

\begin{proposition} \label{prop:maxcut}
    Let $\mathcal{C}$ be a maximal family of nested tight cuts in $G$ and $\mathcal{T}$ any torsoid in $G.$
    Then every vertex of a torso of $\mathcal{C}$ cleaving $\mathcal{T}$ is $(\mathcal{T}, \tightCutsToSets{\mathcal{C}})$-\maxcut.
\end{proposition}

\begin{proof}
    Suppose towards a contradiction that there is a torso $\mathcal{S}$ cleaving $\mathcal{T}$ with a vertex $X$ such that $X$ is not $(\mathcal{T}, \tightCutsToSets{\mathcal{C}})$-\maxcut.
    By \cref{prop:maxcut_vertex} and \cref{prop:maxcut_interval}, there is a $(\mathcal{T},\tightCutsToSets{\mathcal{C}})$-\maxcut tight set $Y$ in $\tightCutsToSets{\mathcal{C}}$ with $\oddIntersections{H}{X} = \oddIntersections{H}{Y}.$
    Then $X \subset Y$ holds and thus there exists a vertex $X'$ of $\mathcal{S}$ that contain an element of $Y \setminus X.$
    Since $\cut{}{Y}$ and $\cut{}{X'}$ are nested, $X'$ is contained in $Y.$

    By construction, $\bigcup \oddIntersections{H}{X} = \bigcup \oddIntersections{H}{Y}$ is contained in $X$, and $Y$ avoids any vertex of $V(H) \setminus \oddIntersections{H}{Y}.$
    Therefore $X'$ does not contain any vertex of $H$, which contradict the definition of cleaving.
\end{proof}

In the following we turn our attention to $H$ and define the specific family of nested tight cuts in $H$ for the proof of \cref{thm:correspond_torso_torsoid}:
\begin{definition}
    \label{def:tight_cut_residents}
    Let $\mathcal{C}$ be a maximal family of nested tight cuts in $G$ and $\mathcal{T}=\Torsoid{H}$ a torsoid in $G.$ We define the set
    \begin{align*}
        \mathcal{C}^{\mathcal{T}} \coloneqq \{ \cut{}{\oddIntersections{H}{X}} :{} &\cut{}{X} \in \mathcal{C} \text{ and }\\
            &\cut{}{X} \text{ resides at an interval or properly at a vertex of $\mathcal{T}$}\}.
    \end{align*}
\end{definition}

\begin{proposition}\label{prop:maximal_set}
    Let $\mathcal{C}$ be a maximal family of nested tight cuts in $G$ and $\mathcal{T}=\Torsoid{H}$ a torsoid in $G.$ Then $\mathcal{C}^{\mathcal{T}}$
    is a maximal family of nested tight cuts in $H.$
\end{proposition}
Note that $\cut{}{X} = \cut{}{\V{G} \setminus X}$ and in particular $\cut{}{\oddIntersections{H}{X}} = \cut{}{\oddIntersections{H}{\V{G} \setminus X}}$ for every $X \subset \V{G}$.
\begin{proof}
    The elements of $\mathcal{C}^{\mathcal{T}}$ are nested, since $\bigcup \oddIntersections{H}{X} \subseteq X$ and $\bigcup \oddIntersections{H}{\V{G} \setminus X} \subseteq \V{G} \setminus X$ for any \pvcut{$\mathcal{T}$} or \icut{$\mathcal{T}$} $X.$
    The cuts in $\mathcal{C}^{\mathcal{T}}$ are indeed tight: for any cut $C \in \mathcal{C}^{\mathcal{T}}$ there is either a \vcut{$\mathcal{T}$} or an \icut{$\mathcal{T}$} $X$ such that $C = \Cut{}{}(\oddIntersections{H}{X}).$
    Then $\oddIntersections{H}{X}$ is a single vertex if $H$ is a \BoB.
    If $H$ is a cycle, $\oddIntersections{H}{X}$ is either a single vertex or an odd interval in $H.$
    This implies that $\Cut{}{}(\oddIntersections{H}{X})$ is a tight set.

    It remains to prove that $\mathcal{C}^{\mathcal{T}}$ is maximal.
    By \cref{prop:maxcut_vertex}, for every vertex $v \in V(H)$ there is a \pvcut{} in $\mathcal{C}$ containing $v$ and therefore all trivial tight cuts of $H$ are contained in $\mathcal{C}^{\mathcal{T}}.$ 
    If $H$ is a \BoB, we are done.
    Thus we can suppose that $H$ is a cycle.
    
    We suppose for a contradiction that there exists a tight cut in $H$ nested with $\mathcal{C}^{\mathcal{T}}$ but not contained in $\mathcal{C}^{\mathcal{T}}.$
    We already observed, that $\mathcal{C}^{\mathcal{T}}$ contains all trivial tight cuts.
    Thus this tight cut is non-trivial and therefore of the form $\cut{}{I}$ for an interval $I$ with $3 \leq |I| \leq |\V{H}|-3.$
    Let $i_1$ be the first and $i_2$ be last element of $I$ (regarding the cyclic order of $\V{H}$).
    We construct a \icut{$\mathcal{T}$} $Y$ with $\cut{}{Y} \in \mathcal{C}$ that holds $\oddIntersections{H}{Y} = I.$
    This contradicts that $\cut{}{I} \notin \mathcal{C}^{\mathcal{T}}.$
    
    For $k \in \Set{1, 2}$ consider all $(\mathcal{T}, \tightCutsToSets{\mathcal{C}})$-\maxcut  \pvcuts{$\mathcal{T}$} and \icuts{$\mathcal{T}$} containing $i_k$ and let $Y_k$ be the subset-maximal one.
    By \cref{prop:maxcut_vertex} and \cref{prop:maxcut_interval}, such $Y_k$ exists.
    It contains all tight sets containing $i_k.$
    As $\cut{}{I}$ is nested with $\cut{}{\oddIntersections{H}{Y_k}} \in \mathcal{C}^{\mathcal{T}}$, $\oddIntersections{H}{Y_k} \subseteq I$ holds for $k \in \Set{1, 2}.$
    Furthermore, since $\cut{}{I}$ is not contained in $\mathcal{C}^{\mathcal{T}}$, $\oddIntersections{H}{Y_k} \subset I$ holds for $k \in \Set{1, 2}.$
    Then $i_1 \cap Y_2 = \emptyset$ and $i_2 \cap Y_1 = \emptyset$, since $Y_1, Y_2$ are \icuts{$\mathcal{T}$} or \vcuts{$\mathcal{T}$}.
    As $\cut{}{Y_1}, \cut{}{Y_2}$ are nested and $i_1 \subseteq Y_1, i_2 \subseteq Y_2$ but $i_1 \cap Y_2 = \emptyset, i_2 \cap Y_1 = \emptyset$, the sets $Y_1, Y_2$ are disjoint.
    
    We consider the set $A\coloneqq \bigcup_{vw \in E(H): v, w \in I} v \cup w \cup \edgefkt{vw}$, which is tight by construction.
    Since $Y_k$ is a \icut{$\mathcal{T}$} or a \pvcut{$\mathcal{T}$} that holds $\oddIntersections{H}{Y_k} \subset I$, $Y_k$ intersects $A$ oddly.
    We set $Y\coloneqq A \cup Y_1 \cup Y_2$, which is tight by \cref{lem:basictight}.
    By construction, $\oddIntersections{H}{Y} = I$ holds.
    We show that $\cut{}{Y}$ is nested with $\mathcal{C}.$
    Then it is contained in $\mathcal{C}$, which gives the desired contradiction.
    
    Suppose that there exists $\cut{}{Z} \in \mathcal{C}$ which crosses $\cut{}{Y}.$
    Then $Z$ crosses $Y.$
    Let $k \in \Set{1, 2}.$
    As $Z$ crosses $Y$, $Z$ is not a subset of $Y_k.$
    Since $Y_k$ is the subset-maximal \pvcut{$\mathcal{T}$} or \icut{$\mathcal{T}$} containing $i_k$, $Z$ is not a superset of $Y_k.$
    Thus $Z$ avoids $Y_k$, as $Z, Y_k$ are nested.
    Therefore it contains elements of both
    \begin{align*}
        \V{G} \setminus Y & \subseteq \bigcup \Brace{V(H) \setminus I} \cup \bigcup_{vw \in E(H): v \notin I} \edgefkt{vw}, \text{ and }\\
        Y \setminus \bigcup \Set{Y_1, Y_2} & \subseteq \bigcup \Brace{I \setminus \Set{i_1, i_2}} \cup \bigcup_{vw \in E(H): v \in I \setminus \Set{i_1, i_2}} \edgefkt{vw}.
    \end{align*}
    The tight set $Z$ clearly can neither be a \vcut{} nor an \pcut{}.
    Since $X$ avoids $i_1 \cup i_2 \subseteq Y_1 \cup Y_2$ it neither can be an \icut{}.
\end{proof}

To obtain the main theorem of this section about the number of torsos cleaving a torsoid it now remains to prove that there is a bijection between the $\kappa_{\mathcal{C}}$-pre-image of a torsoid $\mathcal{T}$, i.e. the set of torsos of $\mathcal{C}$ that cleave $\mathcal{T}$, and the set of torsos of $\mathcal{C}^{\mathcal{T}}$.

\correspondTorsoTorsoidTheorem*

First of all, every torso of $\mathcal{C}$ cleaving $\mathcal{T}$ maps canonically to a torso of $\mathcal{C}^{\mathcal{T}}$:
\begin{proposition}\label{prop:isomorphism_torso}
     Let $\mathcal{C}$ be a maximal family of nested tight cuts in $G$.
     Then every torso $\mathcal{S}$ of $\mathcal{C}$ cleaving a torsoid $\mathcal{T}=\Torsoid{H}$ in $G$ is isomorphic to $\collapse{\Brace{\oddIntersections{H}{X}}_{X \in V(\mathcal{S})}}$.
\end{proposition}

\begin{proof}
    By \cref{obs:refining_partition} any vertex of $\mathcal{S}$ is either a proper \vcut{$\mathcal{T}$} or a \icut{$\mathcal{T}$}. Thus $\Brace{\oddIntersections{H}{X}}_{X \in V(\mathcal{S})}$ is indeed a partition of $V(H).$ It remains to show that for every $X, Y \in V(\mathcal{S})$ there is an edge in $\Cut{G}{X} \cap \Cut{G}{Y}$ if and only if there is $xy \in E(H)$ with $x \in X$ and $y \in Y.$

    Let $X, Y \in V(\mathcal{S})$ be chosen arbitrarily. If there is an edge $e$ in $\Cut{G}{X} \cap \Cut{G}{Y}$, then there is $uv \in E(H)$ such that both endvertices of $e$ are contained in $u \cup \edgefkt{uv} \cup v$ by \cref{torsoid-def-6} and \cref{prop:edges_of_passable_set}. Since any vertex of $\mathcal{S}$ is either a proper \vcut{} or an \icut{}, $X$ and $Y$ each contain exactly one of $u, v$, which proves the forward implication.

    For the backwards implication, note that there is a tight set partition $\mathcal{P}$ corresponding to $\mathcal{T}$ that refines $V(\mathcal{S})$, as any vertex of $\mathcal{S}$ is either a proper \vcut{} or an \icut{}. If there is $xy \in E(H)$ with $x \in X$ and $y \in Y$, let $P_x, P_y \in \mathcal{P}$ such that $x \in P_x$ and $y \in P_y.$ By \cref{thm:inducedtorsoid}, there is an edge between $P_x$ and $P_y$ in $\collapse{\mathcal{P}}.$ This implies that there is an edge between $P_x$ and $P_y$ in $G$ and therefore an edge in $\Cut{G}{X} \cap \Cut{G}{Y}.$
\end{proof}

\begin{proof}[Proof of \cref{thm:correspond_torso_torsoid}]
    Let $\mathcal{C}^{\mathcal{T}}$ be as in \cref{prop:maximal_set}.
    By \cref{prop:number_torso}, the number of torsos of $\mathcal{C}^{\mathcal{T}}$ is $1$ if $H$ is a \BoB and $\frac{n}{2} - 1$ if $H$ is a cycle of length $n.$
    It remains to prove that there is a bijection between the set of torsos of $\mathcal{C}$ cleaving $\mathcal{T}$ and the set of torsos of $\mathcal{C}^{\mathcal{T}}.$
    
    Given a torso $\mathcal{S}$ of $\mathcal{C}$ cleaving $\mathcal{T}$ the graph $\collapse{\Brace{\oddIntersections{H}{X}}_{X \in V(\mathcal{S})}}$ is a torso of $\mathcal{C}^{\mathcal{T}}$ by \cref{prop:isomorphism_torso}.

    Thus the map $\beta$ sending a torso $\mathcal{S}$ of $\mathcal{C}$ cleaving $\mathcal{T}$ to the torso of $\mathcal{C}^{\mathcal{T}}$ at $\Brace{\oddIntersections{H}{X}}_{X \in V(\mathcal{S})}$ is well-defined.
    
    We show that the map $\beta$ is a bijection. Let $\mathcal{S}'$ be an arbitrary torso of $\mathcal{C}^{\mathcal{T}}.$
    For $J \in V(\mathcal{S}')$ let $X_J$ be the unique $(\mathcal{T}, \tightCutsToSets{\mathcal{C}})$-\maxcut tight set with $\oddIntersections{H}{X_J} = J.$
    By \cref{prop:maxcut}, the maximal star $\Brace{X_J}_{J \in \V{\mathcal{S}'}}$ is the only one whose torso is mapped to $\mathcal{S}'$ by $\beta.$
    Therefore $\beta$ is injective. 
    
    For surjectivity we have to show that $\Brace{X_J}_{J \in \V{\mathcal{S}'}}$ is indeed a maximal star of $\mathcal{C}.$
    It remains to prove that $(X_J)_{J \in V(\mathcal{S}')}$ is a  partition of $\V{G}.$
    For $J, K \in \V{\mathcal{S}'}$ the sets $X_J, X_K$ are not contained in each other and $\cut{}{X_J}, \cut{}{X_K}$ are nested, therefore $X_J, X_K$ are disjoint.
    Furthermore, $\bigcup \V{H} \subseteq \bigcup_{J \in \V{\mathcal{S}'}} X_J$ holds.
    We have to prove that $\bigcup_{e \in E(H)} \edgefkt{e} \subseteq \bigcup_{J \in \V{\mathcal{S}'}} X_J$ holds.
    
    Let $vw \in \E{H}$ be arbitrary. If $v, w$ are contained in $X_J$ for some $J \in \V{\mathcal{S}'}$, then also $\edgefkt{vw}$ is contained in $X_J$, as $X_J$ is a \icut{$\mathcal{T}$}.
    Therefore we can assume that there are $K \neq L \in V(\mathcal{S}')$ with $v \in K, w \in L.$
    Suppose towards a contradiction that $Y\coloneqq \edgefkt{vw} \setminus (X_K \cup X_L)$ is nonempty.
    Then $X_K \cup Y = X_K \cup \Brace{\Brace{v \cup \edgefkt{vw}} \cap \Brace{\V{G} \setminus X_L}}$ is a tight set.
    We show that $X_K \cup Y$ is contained in $\tightCutsToSets{\mathcal{C}}$, which contradicts the fact that $X_K$ is $(\mathcal{T}, \tightCutsToSets{\mathcal{C}})$-\maxcut.
    
    For every \vcut{} or \icut{} $Z \in \tightCutsToSets{\mathcal{C}}$ either $\V{G}_Z$ or $\V{G}_{\V{G} \setminus Z}$ is contained in some $J \in V(\mathcal{S}')$, since $\mathcal{S}'$ is a torso of $\mathcal{C}^{\mathcal{T}}.$
    As $X_J$ is $(\mathcal{T}, \tightCutsToSets{\mathcal{C}})$-\maxcut, $Z$ or its complement is contained in $X_J.$
    Therefore no \vcut{} and no \icut{} of $\tightCutsToSets{\mathcal{C}}$ crosses $X_K \cup Y.$
    As no \pcut{$\mathcal{T}$} $Z$ of $\tightCutsToSets{\mathcal{C}}$ at $vw$ crosses $X_K$ or $X_L$, $Z$ is contained in either $X_K$, $X_L$ or $Y$ and thus nested with $X_K \cup Y.$
    Every \pcut{$\mathcal{T}$} at a distinct edge is nested with $X_K \cup Y$, as it is nested with $X_K$ and $Y \subset \edgefkt{vw}.$
    Thus $X_K \cup Y$ is nested with $\tightCutsToSets{\mathcal{C}}$ and therefore contained in $\tightCutsToSets{\mathcal{C}}.$ This gives the desired contradiction.
    
    Indeed, $(X_J)_{J \in V(\mathcal{S}')}$ is a partition of $V(G).$ There are no non-trivial tight cuts in $(X_J)_{J \in V(\mathcal{S}')}$ since there are no non-trivial tight cuts in $\collapse{\V{\mathcal{S}'}}.$ Therefore $(X_J)_{J \in V(\mathcal{S}')}$ is a maximal star of $\mathcal{C}$, which completes the proof that the map $\beta$ is a bijection.
\end{proof}

\bibliographystyle{alpha}
\bibliography{reference}

\end{document}